\documentclass[oneside,10pt]{amsart}
\makeatletter
\@namedef{subjclassname@2020}{%
  \textup{2020} Mathematics Subject Classification}
\makeatother
\usepackage[a4paper,width=160mm,top=27mm,bottom=27mm]{geometry}
\usepackage{cases}
\usepackage{amsmath,amssymb}
\usepackage[hidelinks]{hyperref}
\usepackage{soul}
\usepackage{xcolor}
\usepackage{soul}

\newtheorem{theorem}{Theorem}[section]
\newtheorem{lemma}[theorem]{Lemma}
\newtheorem{definition}[theorem]{Definition}

\newtheorem{corollary}[theorem]{Corollary}

\theoremstyle{definition}
\newtheorem{remx}[theorem]{Remark}
\newenvironment{remark}
  {\pushQED{\qed}\remx}
  {\popQED\endremx}

\newcommand{\la}{\langle}
\newcommand{\ra}{\rangle}
\newcommand{\N}{\mathbb{N}}
\newcommand{\R}{\mathbb{R}}
\newcommand{\T}{\mathbb{T}}

\newcommand{\ld}{\lambda}

\newcommand{\vare}{\varepsilon}

\newcommand{\pt}{\partial}
\newcommand{\bg}{\Big}

\newcommand{\mH}{{E}}

\newcommand{\mK}{{Q}}
\newcommand{\mS}{{S}}
\newcommand{\mI}{{I}}

\newcommand{\Z}{{\mathbb{Z}}}

\newcommand{\mM}{{M}}

\newcommand{\dxy}{\Delta_{x,y}}

\newcommand{\wmM}{\widehat{{M}}}
\newcommand{\wmS}{\widehat{{S}}}
\newcommand{\wmH}{\widehat{{E}}}
\newcommand{\wmI}{\widehat{{I}}}
\newcommand{\wmK}{\widehat{{Q}}}

\newcommand{\wm}{\widehat{m}}
\newcommand{\ba}{\mathbf{a}}
\newcommand{\bb}{\mathbf{b}}
\newcommand{\br}{\mathbf{r}}
\newcommand{\bs}{\mathbf{s}}

\DeclareMathOperator*{\ima}{Im}
\DeclareMathOperator*{\rea}{Re}
\numberwithin{equation}{section}

\begin{document}

\address{Luigi Forcella
\newline \indent  Department of Mathematics
\newline \indent University of Pisa, Italy}
\email{luigi.forcella@unipi.it}

\address{Yongming Luo
\newline \indent
Faculty of Computational Mathematics and Cybernetics
\newline \indent Shenzhen MSU-BIT University, China}
\email{luo.yongming@smbu.edu.cn}

\address{Zehua Zhao
\newline \indent Department of Mathematics and Statistics, Beijing Institute of Technology, Beijing, China
\newline \indent Key Laboratory of Algebraic Lie Theory and Analysis of Ministry of Education, Beijing, China}
\email{zzh@bit.edu.cn}

\title[NLS with combined powers on $\R^d\times\T$]{Solitons, scattering and blow-up for the nonlinear Schr\"odinger equation with combined
power-type nonlinearities on $\R^d\times\T$}
\author{Luigi Forcella, Yongming Luo, and Zehua Zhao}

\begin{abstract}
We investigate the long time dynamics of the nonlinear Schr\"odinger equation (NLS) with combined powers on the waveguide manifold $\R^d\times\T$. Different from the previously studied NLS-models with single power on the waveguide manifolds, where the non-scale-invariance is mainly due to the mixed nature of the underlying domain, the non-scale-invariance of the present model is both geometrical and structural. By considering different combinations of the nonlinearities, we establish both qualitative and quantitative properties of the soliton, scattering and blow-up solutions. As one of the main novelties of the paper compared to the previous results for the NLS with single power, we particularly construct two different rescaled families of variational problems, which leads to an NLS with single power in different limiting profiles respectively, to establish the periodic dependence results.
\end{abstract}


\maketitle

\setcounter{tocdepth}{1}
\tableofcontents

\section{Introduction}
\subsection{Background and motivation}
In this paper, we consider the nonlinear Schr\"odinger equation (NLS)
\begin{align}\label{nls}
(i\pt_t+\Delta_{x,y})u=\mu |u|^p u- |u|^q u
\end{align}
and its corresponding stationary equation
\begin{align}\label{nls2}
-\Delta_{x,y}u+\omega u=-\mu|u|^p u+ |u|^q u
\end{align}
on the waveguide manifold $\R_x^d\times\T_y$ with $\mu\in\{-1,1\}$, $p,q\in(\frac{4}{d},\frac{4}{d-1})$ and $p<q$, where $\T=\T_y$ is the $2\pi$-periodic torus. The NLS arises as a fundamental model in numerous physical scenarios such as the Bose-Einstein condensates and nonlinear optics. For detailed physical background on \eqref{nls}, see e.g. \cite{waveguide_ref_1,waveguide_ref_2,waveguide_ref_3} and the references therein.

Because of its close connection to many mathematical areas, the NLS has also attracted much attention from the mathematical community in recent years. Different from the classical references, where people mainly focused on the NLS-models posed either on an infinite domain (such as the Euclidean space $\R^d$) or on a bounded domain with suitable boundary conditions (e.g. the periodic or Dirichlet boundary conditions), we study in this paper the NLS occupying the mixed semi-periodic domain $\R^d\times\T$, the so-called waveguide manifold. The main interest in studying such models on domains of mixed type is twofold: On the one hand, these models are of physical importance since they arise naturally in actual physical applications, where one or more confinements are applied in order to guide the Schr\"odinger wave to propagate in a designed way\footnote{This also explains where name ``waveguide'' comes from.}. On the other hand, the domain $\R^d\times\T$, or more generally $\R^d\times\T^m$, can be seen as an interpolation between the infinite Euclidean space and a compact torus, thus it is an interesting question whether properties of the Schr\"odinger wave on the sole domains could also be inherited to the waveguide manifolds. 

Recently, there has been a number of papers devoting to the study for NLS on waveguide manifolds, see e.g. \cite{TNCommPDE,TzvetkovVisciglia2016,TTVproduct2014,Ionescu2,HaniPausader,CubicR2T1Scattering,R1T1Scattering,Cheng_JMAA,ZhaoZheng2021,RmT1,YuYueZhao2021,Luo_inter,Luo_energy_crit,Luo_MathAnn_2024,Luo_SIMA_2024,Luo_JGA_2024,Luo_JFA_2024}. It is worth pointing out that only models with single nonlinearity on waveguide were studied in the previously mentioned papers. In many actual physical experiments, however, the NLS-models with combined powers were indeed used which serve as a correction for the theoretically preset models. As one of the most famous examples, the cubic-quintic NLS-model, which corrects the focusing cubic NLS where the predicted collapse phenomenon does not happen in actual experiments, plays a fundamental role in the study of nonlinear optics. Generalizing the cubic and quintic nonlinearities to nonlinearities of arbitrary order leads to the study of \eqref{nls}. For the readers' interest, we refer e.g. to \cite{TaoVisanZhang,MiaoDoubleCrit,killip_visan_soliton,Cheng2020,Carles_Sparber_2021,Murphy2021CPDE,
killip2020cubicquintic,carles2020soliton,Luo_JFA_2022,Luo_DoubleCritical,InteractionCombined} for some recent studies on the NLS with combined powers.

Back to our study, we restrict ourselves in this paper to the so-called \textsl{intercritical} (in other word, strictly larger than mass-critical and smaller than energy-critical) regime $p,q\in(\frac{4}{d},\frac{4}{d-1})$, where the nonlinear effects between the particles could either become weaker and weaker as time goes by and scattering of a solution (i.e. a solution becomes asymptotically linear in the sense of \eqref{def scattering}) might take place, or on the contrary become stronger and stronger which dominate the linear dispersive effects and consequently lead to a possible collapse of the particles. We consider the case where the nonlinearity of higher order $q$ is focusing, and the lower power of order $p$ can be tuned  to be either focusing or defocusing. For simplicity, we consider the case where the prefactor of the $q$-power is $-1$ and the one of the $p$-power is chosen from the set $\{\pm 1\}$.

\subsection{Main results}
We now state the main results of this paper. As our first concern, we shall investigate the existence of the soliton solutions which solve the stationary NLS \eqref{nls2}, since they might be the only observable quantities in actual physical experiments. From a mathematical point of view, the solitons can also be seen as a balance point between the linear and nonlinear effects and thus can be used to characterize a sharp threshold bifurcating the scattering and blow-up solutions, see Theorem \ref{thm-blowup} and Theorem \ref{thm-scattering} below for corresponding rigourous statements.

Mathematically, it is nowadays a standard routine to solve the stationary NLS \eqref{nls2} by appealing to suitable variational methods, namely, one formulates suitable variational problems and looks for optimizers of the corresponding energy functionals on either a Nehari mainifold or in a Sobolev space with given normalized mass. Using the Lagrange multiplier theorem one then immediately infers that those minimizers will also be solving \eqref{nls2}. In the context of the waveguide setting, however, the standard methods can not be directly applied in a straightforward way due to the non-scale-invariance of the torus. For NLS-models with single focusing powers, this issue was solved by the second author in a series of papers \cite{Luo_inter,Luo_energy_crit,Luo_MathAnn_2024,Luo_JGA_2024} by introducing the so-called {\it semivirial-vanishing geometry}. Following the same lines in the previous papers, we adapt the framework of the semivirial-vanishing geometry to the present setting, in order to derive the existence and periodic dependence results of the solitons solutions of the NLS \eqref{nls2} with combined powers.

Before introducing the main results concerning the soliton solutions, several variational quantities are defined in order. Firstly, the mass and energy are defined in the standard way as follows:
\begin{align*}
M(u)=\|u\|^2_{L^2(\R^d\times\T)},\quad
E(u)=\frac12\|\nabla_{x,y}u\|_{L^2(\R^d\times\T)}^2+\frac{\mu}{p+2}\|u\|_{L^{p+2}(\R^d\times\T)}^{p+2}-\frac{1}{q+2}\|u\|_{L^{q+2}(\R^d\times\T)}^{q+2}.
\end{align*}
For $\omega>0$ define the action energy function $S_\omega(u)$ by
\[S_\omega(u)=E(u)+\frac{\omega}{2}M(u).\]
We will also make use of the so-called \textit{semivirial} functional
\[Q(u)=\|\nabla_x u\|_{L^2(\R^d\times\T)}^2
+\frac{\mu p d}{2(p+2)}\|u\|_{L^{p+2}(\R^d\times\T)}^{p+2}
-\frac{q d}{2(q+2)}\|u\|_{L^{q+2}(\R^d\times\T)}^{q+2}.\]
Due to their different geometric nature, we shall define two different variational problems for $\mu=\mp 1$ respectively. More precisely, we define the variational problems $m_c$ and $\gamma_\omega$ by
\begin{align*}
m_c&:=\inf\{E(u):M(u)=c,Q(u)=0\},\\
\gamma_\omega&:=\inf\{S_\omega(u):Q(u)=0\}.
\end{align*}

Our first main result is concerned with the existence of ground states for \eqref{nls2} and reads as follows:
\begin{theorem}[Existence of ground states]\label{thm existence of ground state}
Let $d\geq 1$ and $\frac{4}{d}<p<q<\frac{4}{d-1}$. Then the following statements hold true:
\begin{itemize}
\item[(i)]Let $\mu=-1$. Then for any $c\in(0,\infty)$ the variational problem $m_c$ possesses an optimizer $u_c$ which also solves
    \eqref{nls2} with some $\omega=\omega_c>0$.
\item[(ii)]Let $\mu=1$. Then for any $\omega\in(0,\infty)$ the variational problem $\gamma_\omega$ possesses an optimizer $u_\omega$ which
    also solves \eqref{nls2} with the given $\omega$.
\end{itemize}
\end{theorem}

For Theorem \ref{thm existence of ground state} (i), the semivirial variational strategy of \cite{Luo_inter} is the basic framework, while the coercivity and deformation steps are carried out with both nonlinear terms present; in the final Liouville step we use the combined-power version from \cite{Liouville2}. For Theorem \ref{thm existence of ground state} (ii), the frequency parameter $\omega$ makes the action formulation more flexible, but on the waveguide it is still \textit{a priori} unclear whether the Lagrange multiplier equation obtained directly from the semivirial constraint is elliptic, since the coefficients of $-\Delta_x$ and $-\pt_y^2$ need not have the same sign. Because ellipticity is needed for the regularity underlying the Pohozaev identity, we instead use a mountain-pass construction on a suitable Nehari manifold. See Subsection \ref{sub nehari} for the details.

Another interesting problem is the periodic dependence of the soliton solutions, which is formulated as follows: notice that by the boundedness of a torus we may assume that \eqref{nls} and \eqref{nls2} are constant along the periodic direction. In this case the solitons, for example, will automatically reduce to the ones on $\R^d$. As a natural question, we may ask whether those solitons deduced in Theorem \ref{thm existence of ground state} coincide with the ones on $\R^d$ or not. We refer this kind of questions to as the {\it periodic dependence problems} of the solitons. Such problems were firstly studied by Terracini, Tzvetkov and Visciglia \cite{TTVproduct2014} where the NLS on a generalized product space with a single mass-subcritical\footnote{We note that  in \cite{TTVproduct2014} the exponent of the nonlinearity is mass-subcritical w.r.t. the whole space dimension, 
while in our paper the mass-criticality condition is defined only in term of the Euclidean dimension, therefore the mass-critical threshold exponent in our paper is strictly larger than the one in \cite{TTVproduct2014}.} nonlinearity was studied. By combining the rescaled variational problems introduced in \cite{TTVproduct2014} and the framework of the semivirial-vanishing geometry, similar periodic dependence results for the focusing NLS with at least mass-critical nonlinearity have been recently established by the second author, see \cite{Luo_inter,Luo_energy_crit,Luo_JGA_2024}.

We now state the periodic dependence of the ground state solutions deduced in Theorem \ref{thm existence of ground state}. To formulate the main result we shall still introduce some necessary notations. Let $\wmM(u),\wmH(u),\wmK(u),\wmS_\omega(u)$ be the quantities defined by \eqref{1.18}, \eqref{1.19} and \eqref{1.24}. Define also the variational problems
\begin{align*}
\wm_c&:=\inf\{\wmH(u):u\in H^1(\R^d\times\T),\wmM(u)=c,\wmK(u)=0\},\\
\widehat{\gamma}_\omega&:=\inf\{\wmS_\omega(u):u\in H^1(\R^d\times\T),\wmK(u)=0\}.
\end{align*}
The second main result about the periodic dependence of the ground state solutions given in Theorem \ref{thm existence of ground state} is stated as follows:

\begin{theorem}[Periodic dependence of the ground states]\label{thm threshold mass}
Let $d\geq 1$ and $\frac{4}{d}<p<q<\frac{4}{d-1}$.
\begin{itemize}
\item[(i)] Let additionally $\mu=-1$. Then there exist $0<c_*\leq c^*<\infty$ such that
\begin{itemize}
\item For all $c \in (0,c_*)$ we have $m_{c}<2\pi \wm_{(2\pi)^{-1}c}$ and any minimizer $u_c$ of $m_{c}$ satisfies $\pt_y u_c\neq 0$.

\item For all $c\in(c^*,\infty)$ we have $m_{c}=2\pi \wm_{(2\pi)^{-1}c}$ and any minimizer $u_c$ of $m_{c}$ satisfies $\pt_y u_c=0$.
\end{itemize}

\item[(ii)] Let additionally $\mu=1$. Then there exist $0<\omega_*\leq \omega^*<\infty$ such that
\begin{itemize}
\item For all $\omega \in (0,\omega_*)$ we have $\gamma_{\omega}=2\pi \widehat{\gamma}_{\omega}$ and any minimizer $u_\omega$ of $\gamma_{\omega}$ satisfies $\pt_y u_\omega= 0$.

\item For all $\omega\in(\omega^*,\infty)$ we have $\gamma_{\omega}<2\pi \widehat{\gamma}_{\omega}$ and any minimizer $u^\omega$ of $\gamma_{\omega}$ satisfies $\pt_y u^\omega\neq0$.
\end{itemize}
\end{itemize}
\end{theorem}

In order to prove our results, we follow an approach similar to \cite{TTVproduct2014,Luo_inter,Luo_energy_crit,Luo_JGA_2024}, specifically we shall use suitable rescaled variational problems to prove Theorem \ref{thm threshold mass}. New difficulties however arise due to the presence of the combined powers. More precisely, unlike in \cite{TTVproduct2014} and \cite{Luo_inter,Luo_energy_crit,Luo_JGA_2024}, we are unable to establish the periodic dependence results by using only one single parameterized family of rescaled variational problems, since certain divergence happens when considering the limiting profile of the parameterized problems. We will design well-tailored rescaled variational problems, according to the orders of the nonlinearities, to overcome this problem. 

\begin{remark}
We notice that due to the combined powers we are unable to compare the rescaled variational problem with the unparameterized constant variational problem as in \cite{TTVproduct2014,Luo_inter,Luo_energy_crit,Luo_JGA_2024}, thus it remains an interesting and also challenging open problem whether the thresholds in Theorem \ref{thm threshold mass} (e.g. $c_*$ and $c^*$) will coincide or not.
\end{remark}

 Once the existence of ground states is proved, we have that $u(t)=e^{it}\tilde u$, where $\tilde u$ solves \eqref{nls2}, is a global, non-scattering solution to the time-dependent equation \eqref{nls}. Similar to the single nonlinearity case, it is worth wondering which conditions on the initial datum ensure existence of solution for all times, or conversely lead to formation of singularity in finite time. 

In next theorem, we  establish the existence of  blowing-up solutions for \eqref{nls} and the blow-up rate of such solutions under suitable assumptions. In what follows, and in the rest of the paper, $T_{\max}>0$ and $T_{\min}>0$ are the forward and backward maximal time of existence, respectively, namely the maximal time of existence of the solution $u(t)$ is $I_{\max}=(-T_{\min},T_{\max})$.

\begin{theorem}[Existence of blow-up solutions]\label{thm-blowup}
Let $d\geq 3$, $\frac{4}{d}<p<q<\frac{4}{d-1}$ and $u_0\in H^1(\R^d\times\T)$ satisfy
\[
E(u_0)<m_{M(u_0)}\quad\text{and}\quad Q(u_0)<0
\]
provided $\mu=-1$, or
\[
S_\omega(u_0)<\gamma_\omega\quad\text{and}\quad Q(u_0)<0
\]
provided $\mu=1$ and $\omega>0$. Assume moreover that the initial datum enjoys the following radiality assumption on the non-compact direction: $u_0(x,y)=u_0(|x|,y)$. Then the solution $u$ to \eqref{nls} with $u(0)=u_0$ blows-up in finite time, i.e., $I_{\max}$ is bounded.
\end{theorem}

Our strategy classically combines variational estimates and a localization argument in the virial estimates. Nonetheless, due the anisotropy of the underlying domain, the proof is not straightforward as in the classical Euclidean case, and we make use of a Fourier expansion in the compact direction to carefully estimate the contribution of the localized potential energy terms jointly with the decay of radial Sobolev  functions.

\begin{remark} 
To the best of our knowledge, this paper is the first one dealing with formation of singularity in finite time without the assumption of finite variance in the context of waveguide manifolds. Indeed, as in the classical case of a Euclidean domain, for any $d\geq1$, provided  $u_0\in
H^1(\mathbb R^d\times\mathbb T)\cap L^2(\mathbb R^d\times\mathbb T, |x|^2dxdy)$, we have from \eqref{virial-identity2} the usual identity $V_{|x|^2}''(t)=8Q<0,$ so the Glassey's convexity argument gives a finite time blow-up result directly, see \cite{Luo_inter}.
\end{remark}

 
In light of the results in Theorem \ref{thm-blowup} and  previous remarks, it is worth mentioning that if we remove both the symmetry assumption and the finiteness of the $L^2(\mathbb R\times\mathbb T, |x|^2dxdy)$-norm of the initial datum,  we can prove the so-called grow-up phenomenon, namely we can prove that if the solution is global forward in time. A rigorous and precise statement is given as follows.

\begin{theorem}[Grow-up solutions]\label{thm-growup} Under the hypothesis of Theorem \ref{thm-blowup}, without any restriction on  the space dimension and without  assuming any symmetry, we have the following dichotomy:
either $u(t)$ blows-up in finite time, or the $T_{\max}=+\infty$ (and similarly for $T_{\min}$) and satisfies
$
\displaystyle\limsup_{t\to+\infty}\|\nabla_{x,y} u\|_{ L^2(\R^d\times\T)}=+\infty.
$
\end{theorem}

In order to show Theorem \ref{thm-growup}, we can invoke the results by \cite{DWZ} on intercritical NLS equations posed on $\mathbb R^d$, which rely on the uniform in time control  of $Q(u(t))$ as in Lemma \ref{lem:Q-control}, virial estimates, and an almost finite  speed of propagation. As the proof is very similar to the purely Euclidean setting, we  sketch the main steps in Appendix \ref{App-grow}.

A natural question arising after the proof of the existence of blowing-up solutions would be whether one could describe the blowing-up solutions in a more quantitative way, or in other words whether the rate of the blow-up could be quantified. The last result concerning solutions with finite time singularity formation is the following theorem about an upper bound rate. The proof relies on the estimates established to prove the blow-up results of Theorem \ref{thm-blowup} and a well-known scheme by Merle, Rapha\"el, and Szeftel \cite{MRS}.

\begin{theorem}[An upper bound on the blow-up rate]\label{thm-blowup-rate}
Let the assumptions of Theorem \ref{thm-blowup} be retained. Then for the blow-up solution $u$ to \eqref{nls} with $u(0)=u_0$ given in
Theorem \ref{thm-blowup}, we have:
if $\mu=1$ and $\omega>0$, or if $\mu=-1$, then
\begin{equation} \label{est-blow1}
\int_{t}^{T_{\max}} (T_{\max}-\tau) \|\nabla_{x,y}
u(\tau)\|^2_{L^2(\R^d\times\T)} d\tau\leq C(T_{\max}-t)^{\frac{4(2-q)}{(d-3)q+4}} \hbox{\quad 		
for \quad } t\to T_{\max}^-.
\end{equation}
As a consequence, there exists a time sequence $t_n \to T_{\max}^-$ such that
\begin{equation} \label{blow-rate1}
\|\nabla_{x,y} u(t_n)\|_{L^2(\R^d\times\T)}  \leq C(T_{\max}-t_n)^{-\frac{(d-1)q}{(d-3)q+4}}
\end{equation}
provided $\mu=-1$, and similarly 
provided $\mu=1$ and $\omega>0$. A similar result holds for $-T_{\min}$.
\end{theorem}

At the end of the introductory section, we conclude the analysis concerning the dynamical properties of solutions to \eqref{nls} by considering global solutions, and in particular scattering solutions. As written above, the standing wave $u(t)=e^{it}\tilde u$ is a global non-scattering solution, in the sense that it is not decaying in time. It is worth wondering under which conditions on the initial data, solutions to \eqref{nls} are global, i.e., they exist for every time $t\in \mathbb R$, and moreover how they behave for large times. Specifically, we will establish conditions leading to global well-posedness of the Cauchy problem associated to \eqref{nls}, and we will also prove that a solution \emph{scatters}, namely they behave as a linear Schr\"odinger wave for large time in the energy topology, see \eqref{def scattering} below. The global well-posedness and scattering for small data is classical and is a consequence of perturbation arguments. The next theorem, by using the ground state solutions deduced in Theorem \ref{thm existence of ground state}, ensures scattering for large data. 

\begin{theorem}[Large data scattering]\label{thm-scattering}
Let $d\geq 1$, $\frac{4}{d}<p<q<\frac{4}{d-1}$ and $u_0\in H^1(\R^d\times\T)$ satisfy
$$E(u_0)<m_{M(u_0)}\quad\text{and}\quad Q(u_0)>0$$
provided $\mu=-1$, or
$$S_\omega(u_0)<\gamma_\omega\quad\text{and}\quad Q(u_0)>0$$
provided $\mu=1$ and $\omega>0$. Then the solution $u$ to \eqref{nls} with $u(0)=u_0$ is global and scattering in the
sense that there exist $\phi^{\pm} \in H^1(\mathbb{R}^d\times \mathbb{T})$ such that
\begin{equation}\label{def scattering}
\lim_{t \rightarrow \pm\infty} \|u(t)-e^{it\Delta_{x,y}}\phi^{\pm}\|_{H^1(\mathbb{R}^d\times \mathbb{T})}=0.
\end{equation}
\end{theorem}

In the case of the NLS with single power, the large data scattering result was proved in \cite{Luo_inter} and \cite{Luo_MathAnn_2024} in the cases $d<5$ and $d\geq 5$, by using the {\it concentration compactness} and {\it interaction Morawetz-Dodson-Murphy} methods, respectively. The main reason for such a difference is that the nonlinearity becomes less regular, or more precisely its derivative is no longer Lipschitz, in higher dimensional spaces $d\geq 5$. We shall notice that one may still derive large data scattering results on higher dimensional Euclidean spaces $\R^d$ by appealing to suitable fractional calculus, see e.g. \cite{defocusing5dandhigher}, but so far we don't know whether such fractional calculus is also available on product spaces. Alternatively, such difficulties were overcome by the second author \cite{Luo_MathAnn_2024} by appealing to the interaction Morawetz inequality recently developed by Dodson and Murphy \cite{defocusing5dandhigher,DodsonMurphyNonRadial} which avoids the use of any fractional calculus. At this point, we notice that different from the variational analysis, the proof for Theorem \ref{thm-scattering} is essentially the same as in \cite{Luo_inter,Luo_MathAnn_2024}. For this reason, we shall present a sketch of the proof of Theorem \ref{thm-scattering} in the case $d\geq 5$ based on the interaction Morawetz inequality in Appendix \ref{Sec: scattering}.

\begin{remark}
Notice that in the waveguide setting, even by using the interaction Morawetz inequality we are able to avoid the use of the fractional calculus in the Euclidean space, the fractional derivatives in the periodic direction is still needed. New ideas were then introduced in the recent paper \cite{Luo_MathAnn_2024} by the second author to overcome the additional technical difficulties. 
\end{remark}



The rest of the paper is organized as follows: Sections \ref{sec: 3a} and \ref{sec 3b} are devoted to the proofs of Theorem \ref{thm existence of ground state} and Theorem \ref{thm threshold mass} respectively. In Section \ref{sec 4} we give the proofs for the blow-up results Theorem \ref{thm-blowup} and Theorem \ref{thm-blowup-rate}. In Appendix \ref{Sec: scattering} we give a sketch for the proof of the large data scattering result Theorem \ref{thm-scattering} in the case $d\geq 5$ by using the modern method based on the interaction Morawetz inequality. Finally, a proof for the grow-up result Theorem \ref{thm-growup} will be given in Appendix \ref{App-grow}.

\subsubsection*{Acknowledgment}
Y. Luo was supported by the NSF grant of Guangdong (No. 2024A1515010497), the QB-Program of Guangdong (No. 2024QN11X141) and the NSF grant of China (No. 12301301). Z. Zhao was supported by the NSF grant of China (No. 12101046, 12271032) and the Beijing Institute of Technology Research Fund Program for Young Scholars.

\subsection{Notations and preliminaries}
For simplicity, we ignore in most cases the dependence of the function spaces on their underlying domains and hide this dependence in their
indices. For example $L_x^2=L^2(\R^d)$, $H_{x,y}^1= H^1(\R^d\times \T)$
and so on. However, when the space is involved with time, we still display the underlying temporal interval such as $L_t^pL_x^q(I)$,
$L_t^\infty L_{x,y}^2(\R)$ etc. The norm $\|\cdot\|_p$ is defined by $\|\cdot\|_p:=\|\cdot\|_{L_{x,y}^p}$.

Next, we define the variational quantities, such as mass and energy etc. that will be frequently used in the proof of the main results. For $u\in
H_{x,y}^1$, define
\begin{align}
\mH(u)&:=\frac{1}{2}\|\nabla_{x,y} u\|^2_{2}+\frac{\mu}{p+2}\|u\|^{p+2}_{p+2}-\frac{1}{q+2}\|u\|^{q+2}_{q+2},\label{def of mhu}\\
\mM(u)&:=\|u\|^2_{2},\quad\mS_\omega(u):=\frac{\omega}{2}\mM(u)+\mH(u),\\
\mK(u)&:=\|\nabla_{x} u\|^2_{2}+\frac{\mu p d}{2(p+2)}\|u\|^{p+2}_{p+2}-\frac{ q d}{2(q+2)}\|u\|^{q+2}_{q+2},\\
\mI(u)&:=\frac{1}{2}\|\pt_y u\|_{2}^2+\bg(\frac{1}{2}-\frac{2}{pd}\bg)\|\nabla_x
u\|_{2}^2+\frac{1}{q+2}\bg(\frac{q}{p}-1\bg)\|u\|^{q+2}_{q+2}=\mH(u)-\frac{2}{pd}\mK(u).\label{def of mI}
\end{align}
For $\ld\in[0,\infty]$ (as long as the quantities are well-defined), define
\begin{align}\label{def modified energy}
\mH_{\ld}(u)&:=\frac{\ld}{2}\|\pt_{y} u\|^2_{2}+\frac{1}{2}\|\nabla_{x} u\|^2_{2}
+\frac{\mu\ld^{\frac{p}{q}-1}}{p+2}\|u\|^{p+2}_{p+2}-\frac{1}{q+2}\|u\|_{q+2}^{q+2},\\
\mH^{\ld}(u)&:=\frac{\ld}{2}\|\pt_{y} u\|^2_{2}+\frac{1}{2}\|\nabla_{x} u\|^2_{2}
+\frac{\mu}{p+2}\|u\|^{p+2}_{p+2}-\frac{\ld^{\frac{q}{p}-1}}{q+2}\|u\|_{q+2}^{q+2},\\
\mS_{1,\ld}(u)&:=\frac12\mM(u)+\mH_{\ld}(u),\quad \mS_{1}^{\ld}(u):=\frac12\mM(u)+\mH^{\ld}(u),\\
\mK_{\ld}(u)&:=\|\nabla_{x} u\|^2_{2}
+\frac{\mu \ld^{\frac{p}{q}-1}pd}{2(p+2)}\|u\|^{p+2}_{p+2}-\frac{qd}{2(q+2)}\|u\|_{q+2}^{q+2},\label{def of Q ld}\\
\mK^{\ld}(u)&:=\|\nabla_{x} u\|^2_{2}
+\frac{\mu pd}{2(p+2)}\|u\|^{p+2}_{p+2}-\frac{\ld^{\frac{q}{p}-1}qd}{2(q+2)}\|u\|_{q+2}^{q+2},\\
\mI_{\ld}(u)&:=\frac{\ld}{2}\|\pt_y u\|_{2}^2+\bg(\frac{1}{2}-\frac{2}{pd}\bg)\|\nabla_x u\|_{2}^2
+\frac{1}{q+2}\bg(\frac{q}{p}-1\bg)\|u\|_{q+2}^{q+2}.\label{def of I ld}
\end{align}
For $u\in H_x^1$, define
\begin{align}
\wmH(u)&:=\frac{1}{2}\|\nabla_{x} u\|^2_{L_x^2}+\frac{\mu}{p+2}\|u\|^{p+2}_{L_x^{p+2}}-\frac{1}{q+2}\|u\|^{q+2}_{L_x^{q+2}}\label{1.18}\\
\wmM(u)&:=\|u\|^2_{L_x^2},\quad \wmS_\omega(u):=\wmH(u)+\frac{\omega}{2}\wmM(u),\label{1.19}\\
\wmI(u)&:=\bg(\frac{1}{2}-\frac{2}{pd}\bg)\|\nabla_x u\|_{L_x^2}^2+\frac{1}{q+2}\bg(\frac{q}{p}-1\bg)\|u\|_{L_x^{q+2}}^{q+2},\label{def of wmI}\\
\wmH_\ld(u)&:=\frac{1}{2}\|\nabla_{x}
u\|^2_{L_x^2}+\frac{\mu\ld^{\frac{p}{q}-1}}{p+2}\|u\|^{p+2}_{L_x^{p+2}}-\frac{1}{q+2}\|u\|^{q+2}_{L_x^{q+2}},\\
\wmH^{\ld}(u)&:=\frac{1}{2}\|\nabla_{x} u\|^2_{L_x^2}+\frac{\mu}{p+2}\|u\|^{p+2}_{L_x^{p+2}}
-\frac{\ld^{\frac{q}{p}-1}}{q+2}\|u\|^{q+2}_{L_x^{q+2}},\\
\wmS_{1,\ld}(u)&:=\frac12\wmM(u)+\wmH_{\ld}(u),\quad \wmS_{1}^{\ld}(u):=\frac12\wmM(u)+\wmH^{\ld}(u),\label{1.24}\\
\wmK_\ld(u)&:=\|\nabla_{x} u\|^2_{L_x^2}+\frac{\mu \ld^{\frac{p}{q}-1}p
d}{2(p+2)}\|u\|^{p+2}_{L_x^{p+2}}-\frac{qd}{2(q+2)}\|u\|^{q+2}_{L_x^{q+2}},\\
\wmK^{\ld}(u)&:=\|\nabla_{x} u\|^2_{L_x^2}+\frac{\mu pd}{2(p+2)}\|u\|^{p+2}_{L_x^{p+2}}
-\frac{\ld^{\frac{q}{p}-1}qd}{2(q+2)}\|u\|^{q+2}_{L_x^{q+2}},\\
\wmI^\ld(u)&:=\bg(\frac{1}{2}-\frac{2}{pd}\bg)\|\nabla_x u\|_{L_x^2}^2+\frac{\ld^{\frac{q}{p}-1}}{q+2}\bg(\frac{q}{p}-1\bg)\|u\|_{L_x^{q+2}}^{q+2}.
\end{align}
As convention, the number $\infty^{\frac{p}{q}-1}$ is defined as zero. We also define the sets 
\begin{align}
V(c):=\{u\in S(c):\mK(u)=0\},\label{def vc}
\end{align}

\begin{align}
S(c)&:=\{u\in H_{x,y}^1:\mM(u)=c\},\quad\widehat{S}(c):=\{u\in H_x^1:\wmM(u)=c\}\\
V_\ld(c)&:=\{u\in S(c):\mK_\ld(u)=0\},\quad
V^\ld(c):=\{u\in S(c):\mK^\ld(u)=0\},\\
\widehat{V}_\ld(c)&:=\{u\in \widehat{S}(c):\wmK_\ld(u)=0\},\quad\widehat{V}^\ld(c):=\{u\in \widehat{S}(c):\wmK^\ld(u)=0\}.
\end{align}
and the variational problems
\begin{alignat}{2}
&m_{c,\ld}:=\inf\{\mH_\ld(u):u\in V_\ld(c)\},\quad
&&m_{c}^\ld:=\inf\{\mH^\ld(u):u\in V^\ld(c)\},\label{def of auxiliary problem} \\
&\gamma_{1,\ld}:=\inf\{\mS_{1,\ld}(u):u\in H_{x,y}^1,\,\mK_\ld(u)=0\},\quad
&&\gamma_{1}^{\ld}:=\inf\{\mS_{1}^{\ld}(u):u\in H_{x,y}^1,\,\mK^\ld(u)=0\},\label{1.31}\\
&\wm_{c,\ld}:=\inf\{\wmH_\ld(u):u\in \widehat{V}_\ld(c)\},\quad
&&\wm_c^{\ld}:=\inf\{\wmH^\ld(u):u\in \widehat{V}^\ld(c)\},\label{def of auxiliary problem 2}\\
&\widehat{\gamma}_{1,\ld}:=\inf\{\wmS_{1,\ld}(u):u\in H_{x}^1,\,\wmK_\ld(u)=0\},\quad
&&\widehat{\gamma}_{1}^{\ld}:=\inf\{\wmS_{1}^{\ld}(u):u\in H_{x}^1,\,\wmK^\ld(u)=0\}.\label{1.33}
\end{alignat}
Finally, for a function $u\in H_{x,y}^1$, the scaling operator $u\mapsto u^t$ for $t\in(0,\infty)$ is defined by
\begin{align}\label{def of scaling op}
u^t(x,y):=t^{\frac d2}u(tx,y).
\end{align}

The following useful results for the variational problem $\wm_{c,0}$ are stated in order. For a proof, see e.g.
\cite{Cazenave2003,Jeanjean1997,Bellazzini2013,BellazziniJeanjean2016}.

\begin{lemma}\label{lem wmc property}
The following statements hold true:
\begin{itemize}
\item[(i)]For any $c>0$ the variational problem $\wm_{c,\infty}$ has an optimizer $P_c\in \widehat{S}(c)$. Moreover, $P_c$ satisfies the
    standing wave equation
\begin{align}\label{standing wave on rd}
-\Delta_x P_c+\omega_c P_c=|P_c|^q P_c
\end{align}
with some $\omega_c>0$.
\item[(ii)] Any solution $P_c\in H_x^1$ of \eqref{standing wave on rd} with $\omega_c>0$ is of class $W^{3,p}(\R^d)$ for all
    $p\in[2,\infty)$.
\item[(iii)] Any solution $P_c\in H_x^1$ of \eqref{standing wave on rd} with $\omega_c\geq 0$ satisfies $\wmK_{\infty}(P_c)=0$.
\item[(iv)] The mapping $c\mapsto \wm_{c,0}$ is strictly monotone decreasing and continuous on $(0,\infty)$.
\end{itemize}
\end{lemma}

\section{Existence of normalized ground states: Proof of Theorem \ref{thm existence of ground state}}\label{sec: 3a}

\subsection{Some useful auxiliary lemmas}
Before proving Theorem \ref{thm existence of ground state}, we collect firstly some useful auxiliary lemmas which will be used throughout the paper.

\begin{lemma}[Concentration compactness, \cite{TTVproduct2014}]\label{lemma non vanishing limit}
Let $(u_n)_n$ be a bounded sequence in $H_{x,y}^1$. Assume also that there exists some $\alpha\in(0,\frac{4}{d-1})$ such that
\begin{align}
\liminf_{n\to\infty}\|u_n\|_{\alpha+2}>0.
\end{align}
Then there exist $(x_n)_n\subset \R^d$ and some $u\in H_{x,y}^1\setminus\{0\}$ such that up to a subsequence
\begin{align}
u_n(x+x_n,y)\rightharpoonup u(x,y)\quad\text{weakly in $H_{x,y}^1$}.
\end{align}
\end{lemma}

\begin{lemma}[Scale-invariant Gagliardo-Nirenberg inequality on $\R^d\times\T$, \cite{Luo_inter}]\label{lemma gn additive}
For $\alpha\in(\frac{4}{d},\frac{4}{d-1})$ the following inequality holds for all $u\in H_{x,y}^1$:
\begin{align}
\|u\|_{\alpha+2}^{\alpha+2}\lesssim \|\nabla_x u\|_2^{\frac{\alpha d}{2}}\|u\|_2^{\frac{4-\alpha(d-1)}{2}}
(\| u\|_{2}^{\frac{\alpha}{2}}+\|\pt_y u\|_{2}^{\frac{\alpha}{2}})
\end{align}
\end{lemma}

As an immediate consequence of Lemma \ref{lemma gn additive}, we deduce the following useful properties of a minimizing sequence of the variational problem $m_c$.

\begin{corollary}\label{cor lower bound}
For any $c\in(0,\infty)$ we have $m_c\in(0,\infty)$. Moreover, there exists a bounded sequence $(u_n)_n\subset V(c)$ (recall the definition of $V(c)$ in \eqref{def vc}) such that
$m_c=E(u_n)+o_n(1)$ and
$\liminf_{n\to\infty}\|u_n\|_{q+2}>0$.
\end{corollary}

\begin{proof}
That $m_c<\infty$ follows from $V(c)\neq\varnothing$, the latter being deduced from the scaling properties of the mapping $t\mapsto Q(u^t)$
given in Lemma \ref{monotoneproperty} below. Next, let  $(u_n)_n\subset V(c)$ be a minimizing sequence such that ${E}(u_n)=m_c+o_n(1)$. Then
\begin{align*}
\infty&>m_c+o_n(1)={E}(u_n)={E}(u_n)-\frac{2}{p d}{Q}(u_n)\\
&=\frac{1}{2}\|\pt_y u_n\|_2^2+\bg(\frac{1}{2}-\frac{2}{pd}\bg)\|\nabla_x u_n\|_2^2+\bg(\frac{q}{p}-1\bg)\frac{1}{q+2}\|u_n\|_{q+2}^{q+2}.
\end{align*}
Combining $q>p>4/d$, which in turn implies $\min\{\frac{1}{2}-\frac{2}{p d},\frac{q}{p}-1\}>0$, we infer that $(u_n)_n$ is a bounded
sequence in $H_{x,y}^1$. Now by Lemma \ref{lemma gn additive} and the fact that $p<q<4/(d-1)$ we deduce
\begin{align*}
\|\nabla_xu_n\|_2^2&=\frac{p d}{2(p+2)}\|u_n\|_{p+2}^{p+2}+\frac{q d}{2(q+2)}\|u_n\|_{q+2}^{q+2}\\
&\lesssim \|\nabla_x u_n\|_2^{\frac{p d}{2}}\|u_n\|_2^{\frac{4-p(d-1)}{2}}
(\| u_n\|_{2}^{\frac{p}{2}}+\|\pt_y u_n\|_{2}^{\frac{p}{2}})\\
&\quad +\|\nabla_x u_n\|_2^{\frac{q d}{2}}\|u_n\|_2^{\frac{4-q(d-1)}{2}}
(\| u_n\|_{2}^{\frac{q}{2}}+\|\pt_y u_n\|_{2}^{\frac{q}{2}})\\
&\lesssim \|\nabla_x u_n\|_2^{\frac{p d}{2}}+\|\nabla_x u_n\|_2^{\frac{q d}{2}},
\end{align*}
which combining $q>p>4/d$ implies
\begin{align*}
\liminf_{n\to\infty}(\|u_n\|_{p+2}^{p+2}+\|u_n\|_{q+2}^{q+2})\sim\liminf_{n\to\infty}\|\nabla_x u_n\|_2^2>0.
\end{align*}
Using also interpolation we know that there exists some $\theta\in(0,1)$ such that
\begin{align*}
1\lesssim\|u_n\|_{p+2}^{p+2}+\|u_n\|_{q+2}^{q+2}&\lesssim
\|u_n\|_{q+2}^{q+2}+\|u_n\|_2^{(p+2)(1-\theta)}\|u_n\|_{q+2}^{(p+2)\theta}\nonumber\\
&\lesssim
\|u_n\|_{q+2}^{(p+2)\theta}(1+\|u_n\|_{q+2}^{(q+2)-(p+2)\theta})\lesssim\|u_n\|_{q+2}^{(p+2)\theta}.
\end{align*}
Thus
\begin{align}
\liminf_{n\to\infty}\|u_n\|_{q+2}>0.\label{key gn inq2}
\end{align}

Summing up, we obtain
\begin{align*}
m_c&=\lim_{n\to\infty}\bg(\frac{1}{2}\|\pt_y u_n\|_2^2+\bg(\frac{1}{2}-\frac{2}{p d}\bg)\|\nabla_x u_n\|_2^2
+\bg(\frac{q}{p}-1\bg)\frac{1}{q+2}\|u_n\|_{q+2}^{q+2}\bg)\\
&\gtrsim \liminf_{n\to\infty}\|\nabla_x u_n\|_{q+2}^{q+2}\gtrsim 1,
\end{align*}
which completes the proof.
\end{proof}

\subsection{Dynamical properties of the mappings $t\mapsto \mK(u^t)$ and $c\mapsto m_c$}
We state in this subsection the dynamical properties of the mappings $t\mapsto \mK(u^t)$ and $c\mapsto m_c$, which will play a crucial role in the upcoming proofs.

\begin{lemma}[Property of the mapping $t\mapsto \mK(u^t)$ ]\label{monotoneproperty}
Let $c>0$ and $u\in S(c)$. Then the following statements hold true:
\begin{enumerate}
\item[(i)] $\frac{\partial}{\partial t}\mH(u^t)=t^{-1} Q(u^t)$ for all $t>0$.
\item[(ii)] There exists some $t^*=t^*(u)>0$ such that $u^{t^*}\in V(c)$.
\item[(iii)] We have $t^*<1$ if and only if $\mK(u)<0$. Moreover, $t^*=1$ if and only if $\mK(u)=0$.
\item[(iv)] Following inequalities hold:
\begin{equation*}
Q(u^t) \left\{
\begin{array}{lr}
             >0, &t\in(0,t^*) ,\\
             <0, &t\in(t^*,\infty).
             \end{array}
\right.
\end{equation*}
\item[(v)] $\mH(u^t)<\mH(u^{t^*})$ for all $t>0$ with $t\neq t^*$.
\end{enumerate}
\end{lemma}

\begin{proof}
(i) follows from direct calculation. Now define $y(t):= \frac{\partial}{\partial t}\mH(u^t)$. Then
\begin{align*}
y(t)&=t\|\nabla_x u\|_2^2-\frac{ p d}{2(p+2)}t^{\frac{p d}{2}-1}\|u\|_{p+2}^{p+2}
-\frac{q d}{2(q+2)}t^{\frac{q d}{2}-1}\|u\|_{q+2}^{q+2},\\
y'(t)&=\|\nabla_x u\|_2^2-\frac{2 p d(p d-2)}{4(p+2)}t^{\frac{p d}{2}-2}\|u\|_{p+2}^{p+2}
-\frac{2q d(q d-2)}{4(q+2)}t^{\frac{q d}{2}-2}\|u\|_{q+2}^{q+2}.
\end{align*}
Using $q>p>4/d$ we infer that $y'(0)=\|\nabla_x u\|_2^2>0$, $y'(t)\to -\infty$ as $t\to\infty$ and $y'(t)$ is strictly monotone decreasing on
$(0,\infty)$. Thus there exists some $t_0>0$ such that $y'(t)$ is positive on $(0,t_0)$ and negative on $(t_0,\infty)$. Consequently, we conclude
that $y(t)$ has a zero at $t^*>t_0$, $y(t)$ is positive on $(0,t^*)$ and negative on $(t^*,\infty)$. (ii) and (iv) now follow from the fact
$$y(t)=\frac{\partial \mH(u^t)}{\partial t}=\frac{Q(u^t)}{t}.$$
For (iii), assume first $\mK(u)<0$. Then
$$0>\mK(u)=\frac{Q(u^1)}{1}=y(1),$$
which is only possible as long as $t^*<1$. Conversely, let $t^*<1$. Then using the fact that $y(t)$ is monotone decreasing on $(t^*,\infty)$ we
obtain
$$\mK(u)=y(1)<y(t^*)<0. $$
This completes the proof of (iii). To see (v), integration by parts yields
$$ \mH(u^{t^*})=\mH(u^t)+\int_t^{t^*}y(s)\,ds.$$
Then (v) follows from the fact that $y(t)$ is positive on $(0,t^*)$ and $y(t)$ is negative on $(t^*,\infty)$.
\end{proof}

\begin{lemma}[Property of the mapping $c\mapsto m_c$]\label{monotone lemma}
The mapping $c\mapsto m_c$ is lower semicontinuous and monotone decreasing on $(0,\infty)$.
\end{lemma}

\begin{proof}
Define the functions $f$ and $g$ by
\begin{align*}
f(a,b,c):=\max_{t>0}\{at^2-bt^{\frac{p d}{2}}-ct^{\frac{q d}{2}}\}=:\max_{t>0}g(t,a,b,c).
\end{align*}
We first infer the continuity of the function $f$ on $(0,\infty)^3$. By direct calculation it is easy to verify that for given $a,b,c>0$ there
exists a unique $t_0\in(0,\infty)$ such that $\pt_t g(t_0,a,b,c)=0$ and $\pt_t^2 g(t_0,a,b,c)<0$, hence $f(a,b,c)=g(t_0,a,b,c)$. Thus for given
$(a_0,b_0,c_0)\in(0,\infty)^3$, using the implicit function theorem we may interpret the function $f$ as
\[f(a,b,c)=g(h(a,b,c),a,b,c)\]
with some continuous function $h$ for points $(a,b,c)$ lying in a neighborhood of $(a_0,b_0,c_0)$. This in turn proves the continuity of the
function $f$.

We now show the monotonicity of $c\mapsto m_c$. It suffices to show that for any $0<c_1<c_2<\infty$ and $\vare>0$ we have
\begin{align*}
m_{c_2}\leq m_{c_1}+\vare.
\end{align*}
By the definition of $m_{c_1}$ there exists some $u_1\in V(c_1)$ such that
\begin{align}\label{pert 2}
\mH(u_1)\leq m_{c_1}+\frac{\vare}{2}.
\end{align}
Let $\eta\in C^{\infty}_c(\R^d;[0,1])$ be a cut-off function such that $\eta=1$ for $|x|\leq 1$ and $\eta=0$ for $|x|\geq 2$. For $\delta>0$,
define
\begin{equation*}
\tilde{u}_{1,\delta}(x,y):= \eta(\delta x)\cdot u_1(x,y).
\end{equation*}
Using dominated convergence theorem it is easy to verify that $\tilde{u}_{1,\delta}\to u_1$ in $H_{x,y}^1$ as $\delta\to 0$. Therefore,
\begin{align*}
\|\nabla_{x,y}\tilde{u}_{1,\delta}\|_2&\to \|\nabla_{x,y} u_1\|_2, \\
\|\tilde{u}_{1,\delta}\|_p&\to \| u_1\|_p
\end{align*}
for all $p\in[2,2+\frac{4}{d-1})$ as $\delta\to 0$. Combining the continuity of $f$ we conclude that
\begin{equation}\label{pert 1}
\begin{aligned}
\max_{t>0}\mH(\tilde{u}^t_{1,\delta})
&=\max_{t>0}\bg\{\frac{t^2}{2}\|\nabla_{x}\tilde{u}_{1,\delta}\|_2^2-\frac{t^{\frac{p d}{2}}}{p+2}\|\tilde{u}_{1,\delta}\|_{p+2}^{p+2}
-\frac{t^{\frac{q d}{2}}}{q+2}\|\tilde{u}_{1,\delta}\|_{q+2}^{q+2}
\bg\}+\frac{1}{2}\|\pt_y \tilde{u}_{1,\delta}\|_2^2\\
&\leq \max_{t>0}\bg\{\frac{t^2}{2}\|\nabla_{x}u_1\|_2^2-\frac{t^{\frac{p d}{2}}}{p+2}\|u_1\|_{p+2}^{p+2}
-\frac{t^{\frac{q d}{2}}}{q+2}\|u_1\|_{q+2}^{q+2}
\bg\}+\frac{1}{2}\|\pt_y u_1\|_2^2+\frac{\vare}{4}\\
&=\max_{t>0}\mH( u^t_1)+\frac{\vare}{4}
\end{aligned}
\end{equation}
for sufficiently small $\delta>0$. Now let $v\in C_c^\infty(\R^d)$ with $\mathrm{supp}\,v\subset
B(0,4\delta^{-1}+1)\backslash B(0,4\delta^{-1})$ and define
\begin{equation*}
v_0:= \frac{(c_2-\mM(\tilde{u}_{1,\delta}))^{\frac{1}{2}}}{\mM(v)^{\frac{1}{2}}}\,v.
\end{equation*}
Notice that $v_0$ and $\tilde{u}_{1,\delta}$ have compact supports, which also implies $\mM(v_0)=c_2-\mM(\tilde{u}_{1,\delta})$. Define
\begin{align*}
w_\ld:=\tilde{u}_{1,\delta}+v_0^\ld
\end{align*}
with some to be determined $\ld>0$. Then
\begin{align*}
\|w_\ld\|^p_p=\|\tilde{u}_{1,\delta}\|^p_p+\| v_0^\ld\|^p_p
\end{align*}
for all $p\in [2,2+\frac{4}{d-1})$. Particularly, $\mM(w_\ld)=c_2$. Since $v_0$ is independent of $y\in\T$, we also infer that
\begin{align*}
\|\nabla_{x} w_\ld\|_2&\to\|\nabla_{x} \tilde{u}_{1,\delta}\|_2,\\
\|\pt_y w_\ld\|_2&=\|\pt_y \tilde{u}_{1,\delta}\|_2,\\
\| w_\ld\|_p&\to\| \tilde{u}_{1,\delta}\|_p
\end{align*}
for all $p\in(2,2+\frac{4}{d-1})$ as $\ld\to 0$. Using the continuity of $f$ once again we obtain
\begin{align*}
\max_{t>0}\mH(w^t_\ld)\leq \max_{t>0}\mH(\tilde{u}^t_{1,\delta})+\frac{\vare}{4}
\end{align*}
for sufficiently small $\ld>0$. Finally, combing \eqref{pert 2} and \eqref{pert 1} we infer that
\begin{align*}
m_{c_2}\leq \max_{t>0}\mH(w^t_\lambda)\leq \max_{t>0}\mH(\tilde{u}^t_{1,\delta})+\frac{\varepsilon}{4}\leq
\max_{t>0}\mH(u^t_1)+\frac{\varepsilon}{2}=\mH(u_1)+\frac{\varepsilon}{2}\leq m_{c_1}+\varepsilon,
\end{align*}
which implies the monotonicity of $c\mapsto m_c$ on $(0,\infty)$.

Next, we show the lower semicontinuity of the curve $c\mapsto m_c$. Since $c\mapsto m_c$ is non-increasing, it suffices to show that for any
$c\in(0,\infty)$ and any sequence $c_n\downarrow c$ we have
\begin{align*}
m_c\leq \lim_{n\to \infty}m_{c_n}.
\end{align*}
Let $\vare>0$ be an arbitrary positive number. By the definition of $m_{c_n}$ we can find some $u_n\in V(c_n)$ such that
\begin{align}\label{pert 3}
\mH(u_n)\leq m_{c_n}+\frac{\vare}{2}\leq m_c+\frac{\vare}{2}.
\end{align}
We define $\tilde{u}_n=(c_n^{-1}c)^{\frac{1}{2}} \cdot u_n:=\rho_n u_n$. Then $\mM(\tilde{u}_n)=c$ and $\rho_n\uparrow 1$. Since $u_n\in
V(c_n)$, we obtain
\begin{align*}
m_{c}+\frac{\vare}{2}&\geq m_{c_n}+\frac{\vare}{2}\geq \mH(u_n)=\mH(u_n)-\frac{2}{p d}\mK(u_n)\nonumber\\
&=\frac{1}{2}\|\pt_y u_n\|_2^2+\bg(\frac{1}{2}-\frac{2}{pd}\bg)\|\nabla_x u_n\|_2^2+\bg(\frac{q}{p}-1\bg)\frac{1}{q+2}\|u_n\|_{q+2}^{q+2}.
\end{align*}
Thus $(u_n)_n$ is bounded in $H^1_{x,y}$ and up to a subsequence we infer that there exist $A,B,C\geq 0$ such that
\begin{align*}
\|\nabla_xu_n\|_2^2=A+o_n(1),\quad\|\pt_y u_n\|_2^2=B+o_n(1),\quad\|u_n\|_{\alpha+2}^{\alpha+2}=C_\alpha+o_n(1)
\end{align*}
with $\alpha\in\{p,q\}$. Arguing as in the proof of Corollary \ref{cor lower bound}, we may use the fact $\mK(u_n)=0$ and Lemma \ref{lemma gn
additive} to deduce $A,C_\alpha>0$ and by previous arguments we know that $f$ is continuous at the point $(A,C_p,C_q)$. Using also the fact
that $\rho_n\uparrow 1$ we conclude that
\begin{align*}
m_{c}&\leq \max_{t>0}\mH(\tilde{u}^t_n)
=\max_{t>0}\bg\{\frac{t^2\rho_n^2}{2}\|\nabla_x{u}_{n}\|_2^2
- \sum_{\alpha\in\{p,q\}}\frac{t^{\frac{\alpha d}{2}}\rho_n^{\alpha+2}}{\alpha+2}\|{u}_{n}\|_{\alpha+2}^{\alpha+2}\bg\}
+\frac{\rho_n^2}{2}\|\pt_y u_n\|_2^2\nonumber\\
&\leq\max_{t>0}\bg\{\frac{t^2A}{2}
- \sum_{\alpha\in\{p,q\}}\frac{t^{\frac{\alpha d}{2}}C_\alpha}{\alpha+2}\bg\}
+\frac{1}{2}\|\pt_y u_n\|_2^2+\frac{\vare}{4}\nonumber\\
&\leq \max_{t>0}\bg\{\frac{t^2}{2}\|\nabla_x u_n\|_2^2
- \sum_{\alpha\in\{p,q\}}\frac{t^{\frac{\alpha d}{2}}}{\alpha+2}\| u_n\|_{\alpha+2}^{\alpha+2}\bg\}
+\frac{1}{2}\|\pt_y u_n\|_2^2+\frac{\vare}{2}\nonumber\\
&=\max_{t>0}\mH(u^t_n)+\frac{\vare}{2}=\mH(u_n)+\frac{\vare}{2}\leq m_{c_n}+\vare
\end{align*}
by choosing $n$ sufficiently large. The continuity claim follows from the arbitrariness of $\vare$.
\end{proof}

\subsection{Mountain pass geometry of $E(u)$ on $S(c)$}
As already mentioned in the introductory section, in the waveguide setting, it is {\it a priori} unclear whether the critical points of the variational problem $m_c$ correspond to an elliptic problem, leading to possible failure of applying the Pohozaev's identity to infer that a critical point of $m_c$ is also a solution of the stationary equation \eqref{nls2}. We shall invoke a deformation argument in \cite{Bellazzini2013} to solve this issue.

We begin with defining the mountain pass geometry of $E(u)$ on $S(c)$

\begin{definition}[Mountain pass geometry of $E(u)$ on $S(c)$]\label{definiton of mp geometry}
We say that $\mH(u)$ has a mountain pass geometry on $S(c)$ at the level $\gamma_c$ if there exists some $k>0$ and $\vare\in(0,m_c)$ such that
\begin{align}
\gamma_c:=\inf_{g\in\Gamma(c)}\max_{t\in[0,1]}\mH(g(t))>\max\{\sup_{g\in\Gamma(c)}\mH(g(0)),\sup_{g\in\Gamma(c)}\mH(g(1))\},\label{def of
gammac}
\end{align}
where
\begin{align*}
\Gamma(c):=\{g\in C([0,1];S(c)):g(0)\in A_{k,\vare},\mH(g(1))\leq 0\}
\end{align*}
and
\begin{align*}
A_{k,\vare}:=\{u\in S(c):\|\nabla_{x} u\|_2^2\leq k,\|\pt_y u\|_2^2 \leq 2(m_c-\vare)\}.
\end{align*}
\end{definition}

The following lemma establishes the fact that $m_c$ characterizes the mountain pass level of $E(u)$ on $S(c)$.

\begin{lemma}\label{lemma existence mountain pass}
There exist $k>0$ and $\vare\in(0,m_c)$ such that
\begin{itemize}
\item[(i)]$m_c=\gamma_c$ holds.
\item[(ii)]$\mH(u)$ has a mountain pass geometry on $S(c)$ at the level $m_c$ in the sense of Definition \ref{definiton of mp geometry}.
\end{itemize}
\end{lemma}

\begin{proof}
We firstly prove that by choosing $k$ sufficiently small we have $\mK(u)>0$ for all $u\in A_{k,\vare}$, where $k$ is independent of the choice
of  $\vare\in (0,m_c)$. Indeed, by Lemma \ref{lemma gn additive}, the fact that $\mM(u)=c$, $\|\pt_y u\|_2^2< 2m_c$ for $u\in A_{k,\vare}$ and
$q>p>4/d$ we obtain
\begin{align*}
\mK(u)&=\|\nabla_x u\|_2^2-\frac{p d}{2(p+2)}\|u\|_{p+2}^{p+2}-\frac{q d}{2(q+2)}\|u\|_{q+2}^{q+2}\\
&\geq\|\nabla_x u\|_2^2-C(\|\nabla_x u\|_2^{\frac{p d}{2}}+\|\nabla_x u\|_2^{\frac{q d}{2}})>0.
\end{align*}
as long as $\|\nabla_x u\|_2^2\in(0,k)$ for some sufficiently small $k$. Next we construct the number $\vare$. Arguing as in \eqref{key gn
inq2} we know that there exists some $\beta=\beta(c)>0$ such that if $(u_n)_n\subset V(c)$ is a minimizing sequence for $m_c$, then
\begin{align}
\liminf_{n\to\infty}\bg(\bg(\frac{1}{2}-\frac{2}{p d}\bg)\|\nabla_x u_n\|_2^2\bg)\geq \beta.\label{lower half beta}
\end{align}
We may shrink $\beta$ further such that $\beta<4m_c$. Hence for any minimizing sequence $(u_n)_n$ we must have
\begin{equation}\label{bounded mc sec2}
\begin{aligned}
m_c+\frac{\beta}{4}&\geq \frac{1}{2}\|\pt_y u_n\|_2^2+\bg(\frac{1}{2}-\frac{2}{p d}\bg)\|\nabla_x u_n\|_2^2
+\bg(\frac{q}{p}-1\bg)\frac{1}{q+2}\|u_n\|_{q+2}^{q+2}\\
&>\frac{1}{2}\|\pt_y u_n\|_2^2+\frac{\beta}{2}
\end{aligned}
\end{equation}
for all sufficiently large $n$. Now set $\vare=\frac{\beta}{4}$. For $u\in A_{k,\vare}$, using also Lemma \ref{lemma gn additive}
we infer that
\begin{align*}
\mH(u)&\leq \frac{1}{2}\|\pt_y u\|_2^2+\frac{1}{2}\|\nabla_x u\|_2^2+C(\|\nabla_x u\|_2^{\frac{p d}{2}}+\|\nabla_x u\|_2^{\frac{q
d}{2}})\nonumber\\
&\leq m_c-\vare+\bg(\frac{1}{2}k+Ck^{\frac{p d}{4}}+Ck^{\frac{q d}{4}}\bg).
\end{align*}
We now choose $k=k(\vare)$ sufficiently small (without changing the fact that $\mK(u)>0$ for $u\in A_{k,\vare}$) such that
$\frac{1}{2}k+Ck^{\frac{\alpha d}{4}}<\frac{\vare}{2}$. For this choice of $k$ and $\vare$ we have $\mH(u)<m_c-\vare/2<m_c$ for all $u\in
A_{k,\vare}$ and by definition, (ii) follows immediately from (i).

It is left to show (i). Let $(u_n)_n$ be the given minimizing sequence satisfying \eqref{lower half beta}. Let $u=u_n$ for some (to
be determined) sufficiently large $n\in\N$. For any $\kappa\in(0,\beta/4)$ we can choose $n$ sufficiently large such that $\mH(u)\leq
m_c+\kappa$ and $(\frac{1}{2}-\frac{2}{p d})\|\nabla_x u\|_2^2\geq \beta/2$. Then by \eqref{bounded mc sec2} we know that $\|\pt_y u\|_2^2\leq
2(m_c-\vare)$ for all $\kappa\in(0,\beta/4)$. It is easy to check that $\|\pt_y(u^t)\|_2^2=\|\pt_y u\|_2^2$ for all $t\in(0,\infty)$ and
$\|\nabla_x (u^t)\|_2^2=t^2\|\nabla_x u\|_2^2\to 0$ as $t\to 0$. We then fix some $t_0>0$ sufficiently small such that
$\|\nabla_x(u^{t_0})\|_2^2<k$, which in turn implies that $(u^{t_0})\in A_{k,\vare}$. On the other hand,
\begin{align*}
\mH(u^t)=\frac{1}{2}\|\pt_y u\|_2^2 +\frac{t^2}{2}\|u\|_2^2-\frac{t^{\frac{p d}{2}}}{p+2}\|u\|_{p+2}^{p+2}
-\frac{t^{\frac{q d}{2}}}{q+2}\|u\|_{q+2}^{q+2}\to -\infty
\end{align*}
as $t\to\infty$. We then fix some $t_1$ sufficiently large such that $\mH(u^{t_1})<0$. Now define $g\in C([0,1];S(c))$ by
\begin{align}\label{curve}
g(t):=u^{t_0+(t_1-t_0)t}.
\end{align}
Then $g\in \Gamma(c)$. By definition of $\gamma_c$ and Lemma \ref{monotoneproperty} we have
\begin{align*}
\gamma_c\leq \max_{t\in[0,1]}\mH(g(t))=\mH(u)\leq m_c+\kappa.
\end{align*}
Since $\kappa$ can be chosen arbitrarily small, we conclude that $\gamma_c\leq m_c$. On the other hand, by our choice of $k$ we already know
that for any $g\in \Gamma(c)$ we have $\mK(g(0))>0$. We now prove $\mK(g(1))<0$ for any $g\in \Gamma(c)$. Assume the contrary that there exists
some $g\in\Gamma(c)$ such that $\mK(g(1))\geq 0$. Then
\begin{align*}
0&>\mH(g(1))\geq \frac{1}{2}\|\nabla_x g(1)\|_2^2-\frac{1}{p+2}\|g(1)\|_{p+2}^{p+2}-\frac{1}{q+2}\|g(1)\|_{q+2}^{q+2}\nonumber\\
&\geq \bg(\frac{1}{2}-\frac{2}{p d}\bg)\|\nabla_x g(1)\|_2^2+\bg(\frac{q}{p}-1\bg)\frac{1}{q+2}\|g(1)\|_{p+2}^{p+2}\geq 0,
\end{align*}
a contradiction. Next, by continuity of $g$ there exists some $t\in(0,1)$ such that $\mK(g(t))=0$. Therefore
\begin{align*}
\max_{t\in[0,1]}\mH(g(t))\geq m_c.
\end{align*}
Taking infimum over $g\in\Gamma(c)$ we deduce $\gamma_c\geq m_c$, which completes the desired proof.
\end{proof}

\begin{remark}\label{technical reason}
By technical reason we will also shrink $\vare$ in the Definition \ref{definiton of mp geometry} if necessary such that $\vare\leq
1-\frac{1}{2c_{p,q}}$, where
\[c_{p,q}:=\frac12+\frac12\bg(d-\frac4p\bg)\bg(\frac{2q}{p}+\frac{4}{p}-d\bg)^{-1}.\]
The purpose of this choice of $\vare$ will become clear in the upcoming proof of Lemma \ref{minimizer is solution}. Notice also that
$1-\frac{1}{2c_{p,q}}\in(0,1)$ is equivalent to $c_{p,q}>\frac12$, which by using $p>4/d$ is also equivalent to
$2q/p>d-4/p$. However, this is always satisfied since $p<q<4(d-1)^{-1}$.
\end{remark}

\subsection{Characterization of an optimizer as a standing wave equation}
We now prove that a minimizer of $m_c$ is also a solution of the stationary equation \eqref{nls2}. First we prove a different characterization of $m_c$ that will be more useful in later analysis.
\begin{lemma}\label{3.7lem}
For $c>0$, define
\begin{align}
\tilde{m}_c:=\inf\{\mI(u):u\in S(c),\,\mK(u)\leq 0\},\label{mtilde equal m}
\end{align}
where $\mI(u)$ is defined by \eqref{def of mI}. Then $m_c=\tilde{m}_c$.
\end{lemma}

\begin{proof}
Let $(u_n)_n\subset S(c)$ be a minimizing sequence for the variational problem $\tilde{m}_c$, i.e.
\begin{align}\label{le coz charac}
\mI(u_n)=\tilde{m}_c+o_n(1),\quad
\mK(u_n)\leq 0\quad\forall\,n\in\N.
\end{align}
By Lemma \ref{monotoneproperty} we know that there exists some $t_n\in(0,1]$ such that $\mK(u_n^{t_n})$ is equal to zero. Thus
\begin{align*}
m_c\leq \mH(u_n^{t_n})=\mI(u_n^{t_n})\leq \mI(u_n)=\tilde{m}_c+o_n(1).
\end{align*}
Sending $n\to\infty$ we infer that $m_c\leq \tilde{m}_c$. On the other hand,
\begin{align*}
\tilde{m}_c
\leq\inf\{\mI(u):u\in V(c)\}
=\inf\{\mH(u):u\in V(c)\}=m_c,
\end{align*}
which completes the proof.
\end{proof}

We shall still need to introduce the following preliminary concepts given in \cite{Berestycki1983}. First recall that $S(c)$ is a submanifold
of $H_{x,y}^1$ with codimension $1$. Moreover, the tangent space $T_uS(c)$ for a point $u\in S(c)$ is given by
$$ T_u S(c)=\{v\in H_{x,y}^1:\la u,v\ra_{L_{x,y}^2}=0\}.$$
Denote the tangent bundle of $S(c)$ by $TS(c)$. Next, the energy functional $\mH|_{S(c)}$ restricted to $S(c)$ is a $C^1$-functional on
$S(c)$ and for any $u\in S(c)$ and $v\in T_uS(c)$ we have
\begin{align*}
\la \mH'|_{S(c)}(u),v\ra =\la \mH'(u),v\ra .
\end{align*}
We use $\|\mH'|_{S(c)}(u)\|_*$ to denote the dual norm of $\mH'|_{S(c)}(u)$ in the cotangent space $(T_uS(c))^*$, i.e.
$$\|\mH'|_{S(c)}(u)\|_*:=\sup_{v\in T_u S(c),\,\|v\|_{H^1_{x,y}}\leq 1}|\la \mH'|_{S(c)}(u),v\ra|. $$
Let now
$$\tilde{S}(c):=\{u\in S(c):\mH'|_{S(c)}(u)\neq 0\}.$$
According to \cite[Lem. 4]{Berestycki1983} there exists a locally Lipschitz pseudo gradient vector field $Y:\tilde{S}(c)\to TS(c)$ such that
$Y(u)\in T_uS(c)$ and
\begin{align}\label{pseudo gradient vector}
\|Y(u)\|_{H_{x,y}^1}\leq 2\|\mH'|_{S(c)}(u)\|_*\quad\text{and}\quad \la \mH'|_{S(c)}(u),Y(u)\ra\geq \|\mH'|_{S(c)}(u)\|_*^2
\end{align}
for $u\in \tilde{S}(c)$.

Having all the preliminaries we are ready to prove the claimed statement.

\begin{lemma}\label{minimizer is solution}
For any $c>0$ an optimizer $u_c$ of $m_c$ is a solution of \eqref{nls2} for some $\omega\in\R$.
\end{lemma}

\begin{proof}
We borrow an idea from the proof of \cite[Lem. 6.1]{Bellazzini2013} to show the claim. By Lagrange multiplier theorem we know that $u_c$ solves
\eqref{nls2} is equivalent to $\mH'|_{S(c)}(u)=0$. We hence assume the contrary $\|\mH'|_{S(c)}(u)\|_*\neq 0$, which implies that there exists
some $\delta>0$ and $\mu>0$ such that
\begin{align}\label{gtr mu}
v\in B_{u_c}(3\delta)\Rightarrow\|\mH'|_{S(c)}(v)\|_*\geq \mu,
\end{align}
where $B_{u_c}(\delta):=\{v\in S(c):\|u-v\|_{H_{x,y}^1}\leq \delta\}$. Let $k$ and $\vare$ be given according to Lemma \ref{lemma existence
mountain pass} and Remark \ref{technical reason}. Define
\begin{align*}
\vare_1&:=\frac{1}{4}\bg(m_c-\max\{\sup_{g\in\Gamma(c)}\mH(g(0)),\sup_{g\in\Gamma(c)}\mH(g(1))\}\bg),\\
\vare_2&:=\min\{\vare_1,m_c/4,\mu\delta/4\}.
\end{align*}
We now define the deformation mapping $\eta$ as follows: Let the sets $A,B$ and function $h:S(c)\to [0,\delta]$ be given by
\begin{align*}
A&:= S(c)\cap \mH^{-1}([m_c-2\vare_2,m_c+2\vare_2]),\\
B&:=B_{u_c}(2\delta)\cap \mH^{-1}([m_c-\vare_2,m_c+\vare_2]),\\
h(u)&:=\frac{\delta\,\mathrm{dist}(u,S(c)\setminus A)}{\mathrm{dist}(u,S(c)\setminus A)+\mathrm{dist}(u,B)}.
\end{align*}
Next, we define the pseudo gradient flow $W:S(c)\to H_{x,y}^1$ by
\begin{align*}
W(u):=
\left\{
\begin{array}{ll}
-h(u)\|Y(u)\|^{-1}_{H_{x,y}^1}Y(u),&\text{if $u\in\tilde{S}(c)$},\\
0,&\text{if $u\in S(c)\setminus \tilde{S}(c)$}.
\end{array}
\right.
\end{align*}
One easily verifies that $W$ is a locally Lipschitz function from $S(c)$ to $H_{x,y}^1$. Then by standard arguments (see for instance
\cite[Lem. 6]{Berestycki1983}) there exists a mapping $\eta:\R\times S(c)\to S(c)$ such that $\eta(1,\cdot)\in C(S(c);S(c))$ and $\eta$ solves
the differential equation
\begin{align*}
\frac{d}{dt}\eta(t,u)=W(\eta(t,u)),\quad\eta(0,u)=u
\end{align*}
for any $u\in S(c)$. We now claim that $\eta$ satisfies the following properties:
\begin{itemize}
\item[(i)]$\eta(1,v)=v$ if $v\in S(c)\setminus \mH^{-1}([m_c-2\vare_2,m_c+2\vare_2])$.
\item[(ii)]$\eta(1,\mH^{m_c+\vare_2}\cap B_{u_c}(\delta))\subset\mH^{m_c-\vare_2}$.
\item[(iii)]$\mH(\eta(1,v))\leq \mH(v)$ for all $v\in S(c)$.
\end{itemize}
Here, the symbol $\mH^\kappa$ denotes the set $\mH^\kappa:=\{v\in S(c):\mH(v)\leq \kappa\}$. For (i), by definition we see that $h(v)=0$, thus
$\frac{d}{dt}\eta(t,v)|_{t=0}=W(v)=0$ and $\eta(t,v)\equiv \eta(0,v)=v$. For (iii), using \eqref{pseudo gradient vector} and the non-negativity
of $h$ we obtain
\begin{align*}
\mH(\eta(1,v))&=\mH(v)+\int_0^1\frac{d}{ds}\mH(\eta(s,v))\,ds\nonumber\\
&=\mH(v)-\int_{s\in[0,1],\eta(s,v)\in
\tilde{S}(c)}\la\mH'(\eta(s,v)),h(\eta(s,v))\|Y(\eta(s,v))\|^{-1}_{H_{x,y}^1}Y(\eta(s,v))\ra\,ds\nonumber\\
&\leq \mH(v)-\frac12\int_{s\in[0,1],\eta(s,v)\in \tilde{S}(c)}h(\eta(s,v))\|\mH'|_{S(c)}(\eta(s,v))\|_*\,ds\leq\mH(v).
\end{align*}
It remains to prove (ii). Fix $v\in \mH^{m_c+\vare_2}\cap B_{u_c}(\delta)$. Since the vector field has been normalized in the definition of $W$, for every $w\in S(c)$ we have
\[
\|W(w)\|_{H_{x,y}^1}\leq h(w)\leq\delta.
\]
Hence, using the vector field along the trajectory rather than at the initial point,
\begin{align*}
\|\eta(t,v)-v\|_{H_{x,y}^1}
&\leq \int_0^t\|W(\eta(s,v))\|_{H_{x,y}^1}\,ds\\
&\leq \int_0^t h(\eta(s,v))\,ds\leq t\delta\leq\delta,
\qquad t\in[0,1].
\end{align*}
Since $v\in B_{u_c}(\delta)$, it follows that $\eta(t,v)\in B_{u_c}(2\delta)\subset B_{u_c}(3\delta)$ for all $t\in[0,1]$. Therefore \eqref{gtr mu} gives
\[
\|\mH'|_{S(c)}(\eta(t,v))\|_*\geq\mu
\qquad\text{for every }t\in[0,1].
\]
By (iii), $t\mapsto\mH(\eta(t,v))$ is non-increasing. If $\mH(\eta(t_0,v))\leq m_c-\vare_2$ for some $t_0\in[0,1]$, then (iii) immediately yields $\mH(\eta(1,v))\leq m_c-\vare_2$. Otherwise,
\[
m_c-\vare_2<\mH(\eta(t,v))\leq\mH(v)\leq m_c+\vare_2
\qquad\text{for all }t\in[0,1].
\]
Together with $\eta(t,v)\in B_{u_c}(2\delta)$, this means that $\eta(t,v)\in B$ for every $t\in[0,1]$. Consequently $\operatorname{dist}(\eta(t,v),B)=0$, and the definition of $h$ gives $h(\eta(t,v))=\delta$. Moreover, \eqref{gtr mu} implies that the whole trajectory lies in $\widetilde S(c)$. Using \eqref{pseudo gradient vector} and $\vare_2\leq\mu\delta/4$, we therefore obtain
\begin{align*}
\mH(\eta(1,v))
&\leq \mH(v)-\frac12\int_0^1 h(\eta(s,v))
\|\mH'|_{S(c)}(\eta(s,v))\|_*\,ds\\
&\leq m_c+\vare_2-\frac{\mu\delta}{2}
\leq m_c-\vare_2
\end{align*}
and the proof of (ii) is complete. Next, we recall the function $g$ defined by \eqref{curve} by setting $u=u_c$ therein. We claim that there exist $t_0\ll 1$ and $t_1\gg 1$ such
that $g\in\Gamma(c)$. Indeed, from the proof of Lemma \ref{lemma existence mountain pass} it suffices to show $\|\pt_y u_c\|_2^2\leq
2(m_c-\vare)$. Define the scaling operator $T_\ld$ by
\begin{align}\label{def of t ld}
T_\ld u(x,y):=\ld^{\frac{2}{p}}u(\ld x,y).
\end{align}
Then
\begin{align*}
\|T_\ld(\nabla_x u)\|_2^2&=\ld^{2+\frac4p-d}\|\nabla_x u\|_2^2,\\
\|T_\ld u\|_{p+2}^{p+2}&=\ld^{2+\frac4p-d}\|u\|_{p+2}^{p+2},\quad
\|T_\ld u\|_{q+2}^{q+2}=\ld^{\frac{2q}{p}+\frac4q-d}\|u\|_{q+2}^{q+2},\\
\|T_\ld (\pt_y u)\|_2^2&=\ld^{\frac4p-d}\|\pt_y u\|_2^2,\quad
\|T_\ld u\|_2^2=\ld^{\frac4p-d}\|u\|_2^2,\\
\mK(T_\ld u)&=\ld^{2+\frac4p-d}\mK(u)-\ld^{2+\frac4p-d}\|u\|_{q+2}^{q+2}(\ld^{2(q/p-1)}-1).
\end{align*}
It thus follows that if $\mK(u)=0$ and $\ld\geq 1$, then $\mK(T_\ld u)\leq 0$, where we also invoked the condition $p<q$. Combining Lemma
\ref{monotone lemma}, Lemma \ref{3.7lem} and the fact that $u_c$ is an optimizer of $m_c$ we infer that $\frac{d}{d\ld}I(T_\ld
u_c)|_{\ld=1}\geq 0$, or equivalently
\begin{equation}\label{identity abcda}
\begin{aligned}
\|\pt_y u_c\|_{2}^2&\leq 2\bg(d-\frac{4}{p}\bg)^{-1}\bg(\bg(2+\frac4p-d\bg)\bg(\frac12-\frac{2}{pd}\bg)\|\nabla_x u_c\|_2^2\\
&\quad\quad+\bg(\frac{2q}{p}+\frac4p-d\bg)(q+2)^{-1}\bg(\frac{q}{p}-1\bg)\|u_c\|_{q+2}^{q+2}\bg).
\end{aligned}
\end{equation}
Using $p<q$, \eqref{identity abcda} also implies
\begin{align}\label{identity abcd}
\|\pt_y u_c\|_{2}^2&\leq 2\bg(d-\frac{4}{p}\bg)^{-1}\bg(\frac{2q}{p}+\frac4p-d\bg)\bg(\bg(\frac12-\frac{2}{pd}\bg)\|\nabla_x u_c\|_2^2
+(q+2)^{-1}\bg(\frac{q}{p}-1\bg)\|u_c\|_{q+2}^{q+2}\bg).
\end{align}
Hence
\begin{align*}
m_c&=\mH(u_c)=\mI(u_c)\nonumber\\
&=\frac{1}{2}\|\pt_y u_c\|_{2}^2+\bg(\frac12-\frac{2}{pd}\bg)\|\nabla_x u_c\|_2^2
+(q+2)^{-1}\bg(\frac{q}{p}-1\bg)\|u_c\|_{q+2}^{q+2}\nonumber\\
&\geq c_{p,q}\|\pt_y u_c\|_2^2,
\end{align*}
where the number $c_{p,q}$ is given by Remark \ref{technical reason}. Using the condition $\vare\leq 1-\frac{1}{2c_{p,q}}$ from Remark
\ref{technical reason} we obtain
\begin{align*}
\|\pt_y u_c\|_2^2\leq 2m_c\bg(1-\bg(1-\frac{1}{2c_{p,q}}\bg)\bg)\leq 2m_c(1-\vare),
\end{align*}
as desired. By (i) and Lemma \ref{lemma existence mountain pass} we know that $\eta(1,g(t))\in \Gamma(c)$. We shall finally prove
\begin{align*}
\max_{t\in[0,1]}\mH(\eta(1,g(t)))<m_c,
\end{align*}
which contradicts the characterization $m_c=\gamma_c$ deduced from Lemma \ref{lemma existence mountain pass} and closes the desired proof.
Notice by definition of $g$ and Lemma \ref{monotoneproperty} we have $\mH(g(t))\leq \mH(u_c)=m_c$ for all $t\in[0,1]$. Thus only the following
scenarios can happen:
\begin{itemize}
\item[(a)] $g(t)\in S(c)\setminus B_{u_c}(\delta)$. By (iii) and Lemma \ref{monotoneproperty} (v) we have
\begin{align*}
\mH(\eta(1,g(t)))\leq \mH(g(t))<\mH(u_c)=m_c.
\end{align*}

\item[(b)]$g(t)\in\mH^{m_c-\vare_2} $. By (iii) we have
\begin{align*}
\mH(\eta(1,g(t)))\leq \mH(g(t))\leq m_c-\vare_2<m_c.
\end{align*}

\item[(c)]$g(t)\in\mH^{-1}([m_c-\vare_2,m_c+\vare_2])\cap B_{u_c}(\delta) $. By (ii) we have
\begin{align*}
\mH(\eta(1,g(t)))\leq  m_c-\vare_2<m_c.
\end{align*}
\end{itemize}
This completes the desired proof.
\end{proof}

\subsection{Proof of Theorem \ref{thm existence of ground state} (i)}
Having all the preliminaries we are in a position to prove Theorem \ref{thm existence of ground state} (i).

\begin{proof}[Proof of Theorem \ref{thm existence of ground state} (i)]
We split our proof into three steps.
\subsubsection*{Step 1: Existence of a non-negative optimizer of $m_c$}
By Lemma \ref{3.7lem} we consider equivalently the variational problem $\tilde m_c$. Let $(u_n)_n\subset S(c)$ be a minimizing sequence of $\tilde{m}_c$ satisfying \eqref{le coz charac}. By diamagnetic inequality we know that
the variational problem $\tilde{m}_{c}$ is stable under the mapping $u\mapsto |u|$, thus we may w.l.o.g. assume that $u_n\geq 0$. By Corollary
\ref{cor lower bound} and Lemma \ref{3.7lem} we know that $\tilde{m}_c<\infty$, hence
\begin{align*}
\infty>\tilde{m}_c\gtrsim \mI(u_n)=\bg(\frac{1}{2}-\frac{2}{pd}\bg)\|\nabla_x u_n\|_2^2+\frac{1}{2}\|\pt_y u_n\|_{2}^2
+\frac{1}{q+2}\bg(\frac{q}{p}-1\bg)\|u_n\|^{q+2}_{q+2}.
\end{align*}
Combining $Q(u_n)\leq 0$ and $(u_n)\subset S(c)$ we conclude that $(u_n)_n$ is a bounded sequence in $H_{x,y}^1$. Now using \eqref{key gn
inq2} and Lemma \ref{lemma non vanishing limit} we may find some $H_{x,y}^1\setminus\{0\}\ni u\geq 0$ such that $u_n\rightharpoonup u$ weakly
in $H_{x,y}^1$. By weakly lower semicontinuity of norms we deduce
\begin{align}
\mM(u)=:c_1\in(0,c],\quad\mI(u)\leq \tilde{m}_c.\label{xia jie}
\end{align}
We next show $\mK(u)\leq 0$. Assume the contrary $\mK(u)>0$. By Brezis-Lieb lemma, $\mK(u_n)\leq 0$ and the fact that $L_{x,y}^2$ is a Hilbert
space we infer that
\begin{align*}
\mM(u_n-u)&=c-c_1+o_n(1),\\
\mK(u_n-u)&\leq -\mK(u)+o_n(1).
\end{align*}
Therefore, for all sufficiently large $n$ we know that $\mM(u_n-u)\in(0,c)$ and $\mK(u_n-u)<0$. By Lemma \ref{monotoneproperty} we know that
there exists some $t_n\in(0,1)$ such that $\mK((u_n-u)^{t_n})=0$. Consequently, Lemma \ref{monotone lemma}, Brezis-Lieb lemma and Lemma \ref{3.7lem} yield
\begin{align*}
\tilde{m}_c\leq \mI((u_n-u)^{t_n})<\mI(u_n-u)=\mI(u_n)-\mI(u)+o_n(1)=\tilde{m}_c-\mI(u)+o_n(1).
\end{align*}
Sending $n\to\infty$ and using the non-negativity of $\mI(u)$ we obtain $\mI(u)=0$. This in turn implies $u=0$, which is a contradiction and
thus $\mK(u)\leq 0$. If $\mK(u)<0$, then again by Lemma \ref{monotoneproperty} we find some $s\in(0,1)$ such that $\mK(u^s)=0$. But then using
Lemma \ref{monotone lemma}, Lemma \ref{3.7lem} and the fact $c_1\leq c$
\begin{align*}
\tilde{m}_{c_1}\leq \mI(u^s)<\mI(u)\leq \tilde{m}_c\leq \tilde{m}_{c_1},
\end{align*}
a contradiction. We conclude therefore $\mK(u)=0$

Thus $u$ is a minimizer of $m_{c_1}$. From Lemma \ref{minimizer is solution} we know that $u$ is a solution of \eqref{nls2} and it remains to
show that the corresponding $\omega$ in \eqref{nls2} is positive and $\mM(u)=c$.

\subsubsection*{Step 2: Positivity of $\omega$}
First we prove that $\omega$ is non-negative. Testing \eqref{nls2} with $u$ and followed by eliminating $\|\nabla_x u\|_2^2$ using $\mK(u)=0$
we obtain
\begin{equation}\label{corrected1}
\begin{aligned}
\|\pt_y u\|_2^2+\omega\mM(u)&=
2\bg(d-\frac{4}{p}\bg)^{-1}\bg(\bg(2+\frac4p-d\bg)\bg(\frac12-\frac{2}{pd}\bg)\|\nabla_x u_c\|_2^2\\
&\quad+\bg(\frac{2q}{p}+\frac4p-d\bg)(q+2)^{-1}\bg(\frac{q}{p}-1\bg)\|u_c\|_{q+2}^{q+2}\bg).
\end{aligned}
\end{equation}
Combining \eqref{identity abcda} we infer that $\omega\mM(u)\geq 0$. Since $u\neq 0$ we conclude $\omega\geq 0$. We next show that
$\omega=0$ leads to a contradiction, which completes the proof of Step 2. Assume therefore that $u$ satisfies the equation
\begin{align}\label{liouville}
-\Delta_{x,y}u=u^{p+1}+u^{q+1}.
\end{align}
First consider the case $d\geq 2$. By the Brezis-Kato estimate \cite{BrezisKato} (see also \cite[Lem. B.3]{Struwe1996}) and the local
$L^p$-elliptic regularity (see for instance \cite[Lem. B.2]{Struwe1996}) we know that $u\in W^{2,p}_{\rm loc}(\R^{d+1})$ for all
$p\in[1,\infty)$. Hence by Sobolev embedding we also know that $u$ and $\nabla u$ are of class $L^\infty_{\rm loc}(\R^{d+1})$. Taking $\pt_{j}$
to \eqref{liouville} with $j\in\{1,\dots,d+1\}$ we obtain
\begin{align*}
-\Delta_{x,y}\pt_j u=(p+1)u^{p}\pt_j u+(q+1)u^{q}\pt_j u\in L^\infty_{\rm loc}(\R^{d+1}).
\end{align*}
Hence by applying the local $L^p$-elliptic regularity again we deduce $u\in W^{3,p}_{\rm loc}(\R^{d+1})$ for all $p\in[1,\infty)$.
Consequently, by Sobolev embedding we infer that $u\in C^2(\R^{d+1})$. But then by \cite[Thm. 1.3]{Liouville2} we know that any nonnegative
$C^2$-solution of \eqref{liouville} must be zero, a contradiction.

Next we consider the case $d=1$. For $n\in\N$ let $\phi_n\in C_c^\infty(\R;[0,1])$ be a radially symmetric decreasing function such that
$\phi_n(t)\equiv 1$ on $|t|\leq n$, $\phi_n(t)\equiv 0$ for $|t|\geq n+1$ and $\sup_{n\in\N}\|\phi_n\|_{L_t^\infty}\lesssim 1$. Define also
$$ D_n:=\{x\in\R:|x|\in(n,n+1)\}.$$
Since $u\neq 0$ is non-negative, by monotone convergence theorem we know that
\begin{align*}
\int_{\R\times\T}(u^{p+1}+u^{q+1})\phi_n\,dxdy>0
\end{align*}
for all $n\gg 1$. On the other hand, using the fact that $\mathrm{supp}\,\nabla_x\phi_n\subset D_n$ and H\"older we see that
\begin{equation}\label{contra1}
\begin{aligned}
\int_{\R\times \T}\nabla_x u\cdot\nabla_x \phi_n\,dxdy
&\leq \|\nabla_x u\|_{L^{2}(D_n\times\T)}\|\nabla_x \phi_n\|_{L^2(D_n\times\T)}\\
&\lesssim \|\nabla_x u\|_{L^{2}(D_n\times\T)}((n+1)-n)^{\frac12}\lesssim\|\nabla_x u\|_{L^{2}(D_n\times\T)}.
\end{aligned}
\end{equation}
But then
$\infty>\|\nabla_x u\|_{2}^{2}\geq \sum_{n\geq 1}\|\nabla_x u\|^{2}_{L^{2}(D_n\times\T)}$
yields $\|\nabla_x u\|_{L^{2}(D_n\times\T)}=o_n(1)$. By the fact that $\phi_n$ is independent of $y$ we also know $\int_{\R\times\T}(-\pt_y^2
u)\phi_n\,dxdy=0$.
Summing up, by testing \eqref{nls2} with $\phi_n$ and rearranging terms we obtain
\begin{align*}
0=\int_{\R\times\T}(u^{p+1}+u^{q+1})\phi_n\,dxdy+o_n(1)>0
\end{align*}
for $n$ sufficiently large, a contradiction. This completes the proof of Step 2.

\subsubsection*{Step 3: $\mM(u)=c$ and conclusion}
Finally, we prove $\mM(u)=c$. Assume therefore $c_1<c$. By Lemma \ref{monotone lemma} and \eqref{xia jie} we know that $m_{c_1}$ is a local
minimizer of the mapping $c\mapsto m_c$, which in turn implies that the inequality in \eqref{identity abcda} is in fact an equality. Now using
\eqref{identity abcda} (as an equality) and \eqref{corrected1} we infer that $\omega\mM(u)=0$, which is a contradiction since $\omega>0$ and
$u\neq 0$. We thus conclude $\mM(u)=c$. That $u$ is positive follows immediately from the strong maximum principle. This completes the desired
proof.
\end{proof}

\subsection{Proof of Theorem \ref{thm existence of ground state} (ii)}\label{sub nehari}
In this subsection we give the proof of Theorem \ref{thm existence of ground state} (ii). Again, due to the waveguide setting we are unable to appeal to the Pohozaev's identity to infer that an optimizer of $\gamma_\omega$ is also solving \eqref{nls2}. As in the case of Theorem \ref{thm existence of ground state} (i), we shall use the mountain pass geometry of $\mS_\omega(u)$ on $H_{x,y}^1$ to solve this problem.

\begin{definition}[Mountain pass geometry of $\mS_\omega(u)$ on $H_{x,y}^1$]\label{def 5.1}
We say that $\mS_\omega(u)$ has a mountain pass geometry on $H_{x,y}^1$ at the level $\zeta_\omega$ if there exists some $k>0$ and
\textcolor{blue}{$\vare\in(0,\gamma_\omega)$} such that
\begin{align}
\zeta_\omega:=\inf_{g\in\Lambda(\omega)}\max_{t\in[0,1]}\mS_\omega(g(t))>\max\{\sup_{g\in\Lambda(\omega)}\mS_\omega(g(0))
,\sup_{g\in\Lambda(\omega)}\mS_\omega(g(1))\},
\end{align}
where
\begin{align*}
\Lambda(\omega):=\{g\in C([0,1];H_{x,y}^1):g(0)\in B_{k,\vare},\mS_\omega(g(1))\leq -1\}
\end{align*}
and
\begin{align*}
B_{k,\vare}:=\{u\in H_{x,y}^1:k/2\leq \|\nabla_x u\|_{2}^2\leq k,\,\|\pt_y u\|_2^2+\omega\mM(u)\leq 2(\gamma_\omega-\vare)\}.
\end{align*}
\end{definition}

\begin{lemma}\label{lem 4.2}
The following statements hold:
\begin{itemize}
\item[(i)] There exists some $0<k_0\ll1$ such that for any $\vare\leq \gamma_\omega/2$ and $0<k\leq k_0$ we have $\mK(u)>0$ for all $u\in
    B_{k,\vare}$.
\item[(ii)] If $\mS_\omega(u)\leq-1$, then $\mK(u)<0$.
\end{itemize}
\end{lemma}

\begin{proof}
Using Lemma \ref{lemma gn additive} we have for $u\in B_{k,\vare}$
\begin{align*}
Q(u)&\geq \|\nabla_x u\|_2^2-\frac{qd}{2(q+2)}\|u\|_{q+2}^{q+2}\geq \|\nabla_x u\|_2^2-C_\omega\|\nabla_x u\|_2^{\frac{qd}{2}}.
\end{align*}
Since $q>4/d$, we can find some $k_0\ll 1$ such that for any $k\in(0,k_0]$ the function $z\mapsto z^2-C_\omega z^{\frac{qd}{4}}$ is positive on
$[\frac{k}{2},k]$, and the proof of (i) is complete. For (ii), by direct computation one infers that
\[S_\omega(u)-\frac{2}{pd}Q(u)
=\bg(\frac12-\frac{2}{pd}\bg)\|\nabla_x u\|_2^2+\frac{1}{2}(\|\pt_y u\|_2^2+\omega M(u))
+\bg(\frac{q}{p}-1\bg)\frac{1}{q+2}\|u\|_{q+2}^{q+2}\geq0,\]
from which we conclude that $Q(u)\lesssim S_\omega\leq -1$ and the proof of (ii) is complete.
\end{proof}

\begin{lemma}\label{lem 4.3}
There exist $k>0$ and $\vare\in(0,\gamma_\omega)$ such that
\begin{itemize}
\item[(i)]$\gamma_\omega=\zeta_\omega$ holds.
\item[(ii)]$\mS_\omega(u)$ has a mountain pass geometry on $H_{x,y}^1$ at the level $\gamma_\omega$ in the sense of Definition \ref{def 5.1}
\end{itemize}
\end{lemma}

\begin{proof}
The proof is similar to the one of Lemma \ref{lemma existence mountain pass} and we omit the details here.
\end{proof}

Notice that from Lemma \ref{lemma existence mountain pass} we know that by choosing $k\ll 1$, every curve $g$ from $\Lambda(\omega)$ must go
through the sphere $\{\mK(u)=0\}$. Thus using the mountain pass geometry in the context of \cite[Thm. 4.1]{Ghoussoub1993} (setting $F=\{\mK(u)=0\}$
therein) we obtain the existence of a Palais-Smale sequence for the variational problem $\gamma_\omega$, stated as follows.

\begin{lemma}[Existence of a Palais-Smale sequence]\label{lem ps}
There exists a sequence $(u_n)_n\subset H_{x,y}^1$ such that
$$\mS_\omega(u)=\gamma_\omega+o_n(1),\quad \mathrm{dist}_{H_{x,y}^1}(u_n,\{Q=0\})=o_n(1)
,\quad\|\mS_\omega'(u_n)\|_{H_{x,y}^{-1}}=o_n(1).$$
\end{lemma}

Mimicking the proof of Corollary \ref{cor lower bound} we are able to show the following useful properties of the constructed Palais-Smale
sequence.

\begin{lemma}\label{lem 4.5}
Let $(u_n)_n$ be the Palais-Smale sequence constructed in Lemma \ref{lem ps}. Then $(u_n)_n$ is bounded in $H_{x,y}^1$ and satisfies
$\liminf_{n\to\infty}\|u_n\|_{q+2}>0$.
\end{lemma}

We are now ready to give the complete proof of Theorem \ref{thm existence of ground state} (ii).

\begin{proof}[Proof of Theorem \ref{thm existence of ground state} (ii)]
Let $(u_n)_n$ be the Palais-Smale sequence given by Lemma \ref{lem ps}. Using also Lemma \ref{lem 4.5} and Lemma \ref{lemma non vanishing
limit} we know that $u_n$ converges to some $u\neq 0$ weakly in $H_{x,y}^1$. Using $\|S_\omega'(u_n)\|_{H_{x,y}^1}=o_n(1)$ from Lemma \ref{lem ps} we
know that $u$ solves \eqref{nls2}. It remains to show that $u$ is an optimizer of $\gamma_\omega$.

Define
\begin{align*}
\tilde{\gamma}_\omega:=\inf\{\tilde\mI(u):u\in H_{x,y}^1\setminus\{0\},\,\mK(u)\leq 0\},
\end{align*}
where $\tilde\mI(u)$ is given by
$\tilde\mI(u):=\mS_\omega(u)-\frac{2}{pd}\mK(u)$.
Following the same arguments in Lemma \ref{3.7lem} one easily verifies $\gamma_\omega=\tilde{\gamma}_\omega$. Using $\mathrm{dist}_{H_{x,y}^1}(u_n,\{Q=0\})=o_n(1)$ given in
Lemma \ref{lem ps} we deduce that
\[\tilde\mI(u_n)=\gamma_\omega+o_n(1),\quad \mK(u_n)=o_n(1).\]
By the weakly lower semicontinuity of norms we infer that $\tilde\mI(u)\leq \gamma_\omega$. We still need to show $\mK(u)=0$. Assume first
that $\mK(u)<0$. By Lemma \ref{monotoneproperty} there exists some $t\in(0,1)$ such that $\mK(u^t)=0$. Then
\[\gamma_\omega\leq \tilde\mI(u^{t})<\tilde\mI(u)\leq\gamma_\omega,\]
a contradiction. Suppose now $\mK(u)>0$. Then using the Brezis-Lieb inequality similarly as in the proof of Theorem \ref{thm existence of
ground state} (i) we know that $\mK(u_n-u)<0$ for all $n\gg 1$. But then
\[\gamma_\omega\leq \tilde\mI(u_n-u)=\tilde\mI(u_n)-\tilde\mI(u)+o_n(1)=\gamma_\omega-\tilde\mI(u)+o_n(1),\]
from which we conclude that $\tilde\mI(u)=0$ and consequently $u=0$, a contradiction. This completes the proof.
\end{proof}

\section{Periodic dependence of the ground states: Proof of Theorem \ref{thm threshold mass}}\label{sec 3b}
In this section, we give the proof for Theorem \ref{thm threshold mass}. The following lemma will be playing a crucial role throughout the whole section.

\begin{lemma}\label{lemma auxiliary}
Let $m_{1,\ld},m_{1}^{\ld},\wm_{1,\ld},\wm_1^\ld$ be the quantities defined through \eqref{def of auxiliary problem} and \eqref{def of auxiliary
problem 2}. Then there exist some $0<\ld_*,\ld^*<\infty$ such that
\begin{itemize}
\item For all $\ld \in (0,\ld_*)$ we have $m_{1}^\ld<2\pi \wm_{(2\pi)^{-1}}^\ld$ and any minimizer $u^\ld$ of $m_{1}^\ld$ satisfies $\pt_y u^\ld\neq 0$.

\item For all $\ld\in(\ld^*,\infty)$ we have $m_{1,\ld}=2\pi \wm_{(2\pi)^{-1},\ld}$ and any minimizer $u_\ld$ of $m_{1,\ld}$ satisfies $\pt_y u_\ld= 0$.
\end{itemize}
\end{lemma}

Moreover, we shall mostly concentrate on the proof of Theorem \ref{thm threshold mass} in the variational context where the mass is normalized. The proof for the variational problem, where the frequency $\omega$ is fixed, is quite similar to the former and thus will be sketched out at the end of the proof of Theorem \ref{thm threshold mass} by taking suitable modification into account.

\subsection{Proof of Lemma \ref{lemma auxiliary}}
We firstly characterize $m_{1,\ld}$ in the limit $\ld\to\infty$.

\begin{lemma}\label{auxiliary lemma 1}
We have
\begin{align}\label{limit ld to infty energy sec4}
\lim_{\ld\to\infty}m_{1,\ld}=2\pi\wm_{(2\pi)^{-1},\infty}.
\end{align}
Additionally, let $u_\ld\in V_\ld(1)$ be a positive optimizer of $m_{1,\ld}$ which also satisfies
\begin{align}
-\Delta_x u_\ld-\ld \pt_y^2 u_\ld+\beta_\ld u_\ld=\ld^{\frac{p}{q}-1}|u_\ld|^p u_\ld+|u_\ld|^q u_\ld\quad\text{on $\R^d\times
\T$}\label{vanishing 3 sec4}
\end{align}
for some $\beta_\ld>0$ (whose existence is guaranteed by Theorem \ref{thm existence of ground state}). Then
\begin{align}
\lim_{\ld\to\infty}\ld\|\pt_y u_\ld\|_2^2=0.\label{vanishing sec4}
\end{align}
\end{lemma}

\begin{proof}
It suffices to restrict ourselves to the case $\ld>1$. By assuming that a candidate in $V_\ld(1)$ is independent of $y$ we already conclude
\begin{align}
m_{1,\ld}\leq 2\pi \wm_{(2\pi)^{-1},\ld}\leq 2\pi\wm_{(2\pi)^{-1},\infty}<\infty.\label{upper bound sec4}
\end{align}
To see the second inequality in \eqref{upper bound sec4}, we may simply take $(v_n)_n\subset \wmK_{\infty}((2\pi)^{-1})$ that approaches
$\wm_{(2\pi)^{-1},\infty}$. This particularly yields $\wmK_\ld(v_n)<0$ for all $n\in\N$. By a similar argument as the one given in the proof of
Lemma \ref{3.7lem} it follows
\[\wm_{(2\pi)^{-1},\ld}\leq \wmI(v_n)=\wm_{(2\pi)^{-1},\infty}+o_n(1)\]
The claim follows by sending $n\to\infty$.

Next we prove
\begin{align}
\lim_{\ld\to\infty}\|\pt_y u_\ld\|_2^2=0.\label{vanishing1 sec4}
\end{align}
Suppose that \eqref{vanishing1 sec4} does not hold. Then we must have
\begin{align*}
\lim_{\ld\to\infty}\ld\|\pt_y u_\ld\|_2^2=\infty.
\end{align*}
Since $\mK_\ld(u_\ld)=0$ and $q>p>4/d$,
\begin{equation}\label{contradiction sec4}
\begin{aligned}
m_{1,\ld}&=\mH_\ld(u_\ld)-\frac{2}{p d}\mK_{\ld}(u_\ld)=\mI_\ld(u_\ld)\\
&=\frac{\ld}{2}\|\pt_y u_\ld\|_{2}^2+\bg(\frac{1}{2}-\frac{2}{pd}\bg)\|\nabla_x u_\ld\|_{2}^2
+\frac{1}{q+2}\bg(\frac{q}{p}-1\bg)\|u_\ld\|_{q+2}^{q+2}\\
&\geq \frac{\ld}{2}\|\pt_y u_\ld\|_2^2\to\infty
\end{aligned}
\end{equation}
as $\ld\to\infty$, which contradicts \eqref{upper bound sec4} and in turn proves \eqref{vanishing1 sec4}. Using \eqref{upper bound sec4} and
\eqref{contradiction sec4} we infer that
\begin{align}\label{upper bound 2 sec4}
\|\nabla_x u_\ld\|_2^2\leq C  m_{1,\ld}\leq  2\pi C\wm_{(2\pi)^{-1},\infty}<\infty,
\end{align}
where $C$ is some positive constant independent of $\ld$. Therefore $(u_\ld)_\ld$ is a bounded sequence in $H_{x,y}^1$, whose weak limit is denoted by $u$. Arguing similarly as in the proof of Corollary \ref{cor lower bound} and using the fact that $\ld^{\frac{p}{q}-1}\to 0$ as $\ld\to\infty$ we infer that
\begin{align*}
\liminf_{\ld\to\infty}\|u_\ld\|_{q+2}^{q+2}=
\liminf_{\ld\to\infty}(\ld^{\frac{p}{q}-1}\|u_\ld\|_{p+2}^{p+2}+\|u_\ld\|_{q+2}^{q+2})\sim\liminf_{\ld\to\infty}\|\nabla_x u_\ld\|_2^2>0.
\end{align*}
Hence by Lemma \ref{lemma non vanishing limit} we may also assume that $u\neq 0$. Using
\eqref{vanishing1 sec4} we know that $u$ is independent of $y$ and thus $u\in H_x^1$. Moreover, applying the weakly lower semicontinuity of norms we
know that $\wmM(u)\in(0,(2\pi)^{-1}]$. On the other hand, using $\mK(u_\ld)=0$, $\mM(u_\ld)=1$,
$\ld^{\frac{p}{q}-1}<1$ for $\ld>1$ and H\"older we obtain from \eqref{vanishing 3 sec4} that
\begin{align*}
\beta_\ld\lesssim 1+\|u_\ld\|_{q+2}^{q+2}.
\end{align*}
Thus $(\beta_\ld)_\ld$ is a bounded sequence in $(0,\infty)$, whose limit is denoted by $\beta$. We now test \eqref{vanishing 3 sec4} with
$\phi\in C_c^\infty(\R^d)$ and integrate both sides over $\R^d\times\T$. Notice particularly that the term $\int_{\R^d\times\T}  \pt_y^2 u_\ld
\phi\,dxdy=0$ for any $\ld>0$ since $\phi$ is independent of $y$. Using the weak convergence of $u_\ld$ to $u$ in $H_{x,y}^1$ and
$\lim_{\ld\to\infty}\ld^{\frac{p}{q}-1}=0$, by sending $\ld\to\infty$ we obtain
\begin{align}
-\Delta_x u+\beta u=|u|^q u\quad\text{in $\R^d$}.\label{vanishing 4 sec4}
\end{align}
In particular, by Lemma \ref{lem wmc property} we know that $\wmK_{\infty}(u)=0$. Combining the weakly lower semicontinuity of
norms and \eqref{upper bound sec4} we deduce
$$2\pi\wmH_{\infty}(u)=2\pi\wmI(u)\leq\liminf_{\ld\to\infty}\mI_\ld(u_\ld)= \liminf_{\ld\to\infty}\mH_\ld(u_\ld)
=\liminf_{\ld\to\infty}m_{1,\ld}\leq 2\pi
\wm_{(2\pi)^{-1},\infty}.$$
However, by Lemma \ref{lem wmc property} the mapping $c\mapsto \wm_{c,\infty}$ is strictly monotone decreasing on $(0,\infty)$, from which we
conclude that $\wmM(u)=(2\pi)^{-1}$ and $u$ is an optimizer of $\wm_{(2\pi)^{-1},\infty}$. Finally, using the weakly lower semicontinuity of
norms we conclude
\begin{equation}\label{vanishing 5 sec4}
\begin{aligned}
m_{1,\ld}&=\mI_\ld
=\frac{\ld}{2}\|\pt_y u_\ld\|_{2}^2+\bg(\frac{1}{2}-\frac{2}{pd}\bg)\|\nabla_x u_\ld\|_{2}^2
+\frac{1}{q+2}\bg(\frac{q}{p}-1\bg)\|u_\ld\|_{q+2}^{q+2}\\
&\geq \bg(\frac{1}{2}-\frac{2}{p d}\bg)\|\nabla_x u_\ld\|_2^2+\frac{1}{q+2}\bg(\frac{q}{p}-1\bg)\|u_\ld\|_{q+2}^{q+2}\\
&\geq 2\pi\bg(\bg(\frac{1}{2}-\frac{2}{p d}\bg)\|\nabla_x u\|_{L_x^2}^2+\frac{1}{q+2}\bg(\frac{q}{p}-1\bg)\|u\|_{L_x^{q+2}}^{q+2}
\bg)+o_\ld(1)\\
&=2\pi \wmH_{\infty}(u)+o_\ld(1)\geq 2\pi\wm_{(2\pi)^{-1},\infty}+o_\ld(1).
\end{aligned}
\end{equation}
Letting $\ld\to\infty$ and taking \eqref{upper bound sec4} into account we conclude \eqref{limit ld to infty energy sec4}. Finally,
\eqref{vanishing sec4} follows directly from the computation in \eqref{vanishing 5 sec4} without neglecting $\ld\|u_\ld\|_2^2$ therein. This completes the desired
proof.
\end{proof}

\begin{lemma}\label{strong convergence u ld}
Let $u_\ld$ and $u$ be the functions given in the proof of Lemma \ref{auxiliary lemma 1}. Then $u_\ld\to u$ strongly in $H_{x,y}^1$.
\end{lemma}

\begin{proof}
This simply follows from the observation that in the proof of Lemma \ref{auxiliary lemma 1}, all the inequalities involving the weakly lower
semicontinuity of norms are in fact equalities. Hence $\|u_\ld\|_{H_{x,y}^1}\to \|u\|_{H_{x,y}^1}=\sqrt{2\pi}\|u\|_{H_x^1}$ as $\ld\to\infty$, which
in turn implies the strong convergence of $u_\ld$ to $u$ in $H_{x,y}^1$.
\end{proof}

The next rigidity lemma is analogous in spirit to \cite[Lem. 3.5]{TTVproduct2014}, but we include the argument in the present combined-power setting in order to verify explicitly that the additional term $\ld^{\frac{p}{q}-1}|u_\ld|^p u_\ld$ does not affect the conclusion.

\begin{lemma}[\cite{TTVproduct2014}]\label{lemma no dependence}
There exists some $\ld_0$ such that $\pt_y u_\ld=0$ for all $\ld>\ld_0$.
\end{lemma}

\begin{proof}
Set $D=d+1$ and $z_\ld:=\pt_yu_\ld$. We first record a uniform form estimate for the nonlinear potentials. For each $r\in\{p,q\}$ and every $\eta>0$, there exists $C_\eta>0$, independent of all sufficiently large $\ld$, such that
\begin{align}
\int_{\R^d\times\T}|u_\ld|^r|\phi|^2\,dxdy
\leq \eta\|\nabla_{x,y}\phi\|_2^2+C_\eta\|\phi\|_2^2,
\qquad \phi\in H^1(\R^d\times\T).
\label{uniform form bound rigidity}
\end{align}
Indeed, by Lemma \ref{strong convergence u ld}, $u_\ld\to u$ strongly in $H^1(\R^d\times\T)$. If $D\geq3$, then
\[
2<\frac{rD}{2}<\frac{2D}{D-2}
\]
for $r\in\{p,q\}$, because $\frac4d<r<\frac4{d-1}$. Hence $|u_\ld|^r\to|u|^r$ strongly in $L^{D/2}$. Splitting $|u_\ld|^r$ into a bounded part and an $L^{D/2}$-small tail, uniformly for large $\ld$, and then applying H\"older together with the Sobolev embedding $H^1\hookrightarrow L^{2D/(D-2)}$ gives \eqref{uniform form bound rigidity}. If $D=2$, the same argument uses $H^1(\R\times\T)\hookrightarrow L^s(\R\times\T)$ for every finite $s$, choosing the integrability exponent of $|u_\ld|^r$ strictly larger than one.

For each fixed $\ld$, standard elliptic regularity (equivalently, a difference-quotient argument in the periodic variable) allows us to differentiate \eqref{vanishing 3 sec4} with respect to $y$. We obtain
\begin{align*}
-\Delta_x z_\ld-\ld\pt_y^2z_\ld+\beta_\ld z_\ld
=(p+1)\ld^{\frac pq-1}|u_\ld|^p z_\ld+(q+1)|u_\ld|^q z_\ld.
\end{align*}
Testing this equation against $z_\ld$ gives
\begin{align}
\|\nabla_xz_\ld\|_2^2+\ld\|\pt_yz_\ld\|_2^2+\beta_\ld\|z_\ld\|_2^2
={}&(p+1)\ld^{\frac pq-1}\int|u_\ld|^p|z_\ld|^2
 +(q+1)\int|u_\ld|^q|z_\ld|^2.
\label{rigidity differentiated identity}
\end{align}
Since $p<q$, we have $\ld^{\frac pq-1}\leq1$ for $\ld\geq1$. Applying \eqref{uniform form bound rigidity} to the right-hand side of \eqref{rigidity differentiated identity}, with $\eta$ chosen sufficiently small, yields a constant $C>0$, independent of all sufficiently large $\ld$, such that
\begin{align}
\|\nabla_xz_\ld\|_2^2+\ld\|\pt_yz_\ld\|_2^2+\beta_\ld\|z_\ld\|_2^2
\leq \frac12\|\nabla_{x,y}z_\ld\|_2^2+C\|z_\ld\|_2^2.
\label{rigidity form estimate}
\end{align}
For almost every $x$, the periodic derivative $z_\ld(x,\cdot)=\pt_yu_\ld(x,\cdot)$ has zero mean on $\T$. The periodic Poincar\'e inequality therefore gives
\[
\|z_\ld\|_2^2\leq\|\pt_yz_\ld\|_2^2.
\]
Using this in \eqref{rigidity form estimate} and recalling that $\beta_\ld>0$, we obtain
\[
\frac12\|\nabla_xz_\ld\|_2^2+
\Big(\ld-C-\frac12\Big)\|\pt_yz_\ld\|_2^2\leq0.
\]
Consequently, if $\ld>C+\frac12$, both terms on the left-hand side are nonnegative and hence vanish. In particular $z_\ld=\pt_yu_\ld=0$. Taking, for instance, $\ld_0=C+1$ proves the lemma.
\end{proof}

After having all the preliminaries, we are now able to give the proof of Lemma \ref{lemma auxiliary}.

\begin{proof}[Proof of Lemma \ref{lemma auxiliary}]
Define
\begin{align*}
\ld^*:=\inf\{\tilde{\ld}\in(0,\infty):m_{1,\ld}=2\pi \wm_{(2\pi)^{-1},\ld}\,\forall\,\ld\geq \tilde{\ld}\}.
\end{align*}
From Lemma \ref{lemma no dependence} we already know that $\ld^*\in(0,\infty)$ and the second part of Lemma \ref{lemma auxiliary} holds for the
defined number $\ld^*$. It is left to construct the positive number $\ld_*\in(0,\infty)$ as required in Lemma \ref{lemma auxiliary}.

We first show for any $c\in(0,\infty)$ it holds $\lim_{\ld\to 0}\wm_c^{\ld}=\wm_c^{0}$. Denote by $U^\ld$ a minimizer of $\wm_c^{\ld}$ whose
existence can be deduced by using a similar proof as the one of Theorem \ref{thm existence of ground state}. In particular, $\wmK^\ld(U^0)\leq
\wmK^0(U^0)=0$, which combining Lemma \ref{3.7lem} implies
$$0\leq \sup_{\ld\in(0,1]}\wm_c^{\ld}\leq \sup_{\ld\in(0,1]}\wmI^\ld(U^0)\in(0,\infty). $$
Hence $(\wm_c^{\ld})_{\ld\in(0,1)}$ is a bounded sequence. Using the arguments given in the proof of Corollary \ref{cor lower bound} it follows
that $(U^{\ld})_{\ld\in(0,1)}$ is a bounded sequence in $H_x^1$ with $\liminf_{\ld\to 0}\|U^\ld\|_{L_x^{p+2}}>0$. We may now use the completely
same arguments given in the proof of Theorem \ref{thm existence of ground state} to show that (up to a subsequence) $U^\ld$ converges to some
$\tilde{U}^0$ strongly in $H_x^1$ with $\tilde{U}^0$ being a minimizer of $\wm_c^{0}$. This in turn implies the desired claim.

To proceed, we next construct an auxiliary function $\rho$ as follows: Let $a\in(0,\pi)$ such that
$a>\pi-3\pi\bg(\frac{3}{p+3}\bg)^{\frac2p}$. This is always possible for $a$ sufficiently close to $\pi$. Then we define $\rho$ by
\begin{align*}
\rho(y)=\left\{
\begin{array}{ll}
0,&y\in[0,a]\cup[2\pi-a,2\pi],\\
(\pi-a)^{-1}\bg(\frac{p+3}{3}\bg)^{\frac{1}{p}}(y-a),&y\in[a,\pi],\\
(\pi-a)^{-1}\bg(\frac{p+3}{3}\bg)^{\frac{1}{p}}(2\pi-a-y),&y\in[\pi,2\pi-a].
\end{array}
\right.
\end{align*}
By direct computation and $q>p$ one easily verifies that $\rho\in H_y^1$ and
\begin{align}\label{new1}
\|\rho\|_{L_y^2}^2=\|\rho\|_{L_y^{p+2}}^{p+2}<\min\{2\pi,\|\rho\|_{L_y^{q+2}}^{q+2}\}.
\end{align}
Now let $P^\ld\in H_x^1$ be an optimizer of $\wm_{\|\rho\|_{L_y^2}^{-2}}^\ld$. Using arguments as in the proof of Corollary \ref{cor lower bound} one easily verifies that $(P^\ld)_{\ld\in[0,1)}$ is a bounded sequence in $H_x^1$. By Lemma \ref{lem wmc property}, the mapping $c\mapsto\wm_c^0$ is
strictly decreasing on $(0,\infty)$. Hence $\|\rho\|_{L_y^2}^{-2}>(2\pi)^{-1}$ implies
$\delta:=\wm_{(2\pi)^{-1}}^0-\wm_{\|\rho\|_{L_y^2}^{-2}}^0>0$. Next, define $\psi^\ld(x,y):=\rho(y)P^\ld(x)$. Then $(\psi^\ld)_{\ld\in(0,1]}$ is a bounded sequence in $H_{x,y}^1$ with $\mM(\psi^\ld)=\|\rho\|_{L_y^2}^2\wmM(P^\ld)=1$. For given $\ld$ let $t^\ld\in(0,\infty)$ be given such that $\mK^\ld((\psi^\ld)^{t^\ld})=0$, where
$(\psi^\ld)^{t^\ld}$ is defined by \eqref{def of scaling op}. Using \eqref{new1} it follows $\mK^0(\psi^0)=0$. This, in conjunction with standard
continuity arguments, implies $t^\ld\to 1$ as $\ld\to 0$. Using also $\lim_{\ld\to 0}\wm_c^{\ld}=\wm_c^{0}$ for any $c\in(0,\infty)$
we conclude that
\begin{align*}
m_1^\ld&\leq \mH^\ld((\psi^\ld)^{t^\ld})\leq \|\rho\|_{L_y^2}^2 \wmH^\ld ((P^\ld)^{t^\ld})+o_\ld(1)\\
&=\|\rho\|_{L_y^2}^2\wm_{\|\rho\|_{L_y^2}^{-2}}^\ld+o_{\ld}(1)=\|\rho\|_{L_y^2}^2\wm_{\|\rho\|_{L_y^2}^{-2}}^0+o_\ld(1)\\
&\leq 2\pi(\wm_{(2\pi)^{-1}}^0-\delta)+o_\ld(1)=2\pi\wm_{(2\pi)^{-1}}^\ld-2\pi\delta+o_\ld(1)
\end{align*}
as $\ld\to 0$. This implies $m_{1}^\ld<2\pi \wm_{(2\pi)^{-1}}^\ld$ for all sufficiently small $\ld$.

Finally, we borrow an idea from \cite{GrossPitaevskiR1T1} to show that any minimizer of $m_{1,\ld}$ for $\ld>\ld_*$ must be $y$-independent. Assume the contrary that an optimizer $u_\ld$ of $m_{1,\ld}$ satisfies $\|\pt_y u_\ld\|_2^2\neq 0$. Since $\ld>\ld_*$, there exists some $\kappa$ strictly lying between $\ld_*$ and $\ld$. Then
\begin{align*}
2\pi \wm_{(2\pi)^{-1},\kappa}=m_{1,\kappa}\leq \mH_{\kappa}(u_\ld)=\mH_{\ld}(u_\ld)+\frac{\kappa-\ld}{2}\|\pt_y u_\ld\|_2^2<\mH_{\ld}(u_\ld)=m_{1,\ld}=2\pi \wm_{(2\pi)^{-1},\ld}.
\end{align*}
Nevertheless, using the characterization \eqref{mtilde equal m} for $\wm_{c,\ld}$ one easily deduces that $\wm_{(2\pi)^{-1},\kappa}\geq \wm_{(2\pi)^{-1},\ld}$ and we hence obtain a contradiction. This completes the desired proof.
\end{proof}

\subsection{Proof of Theorem \ref{thm threshold mass}}
We are in a final position to prove Theorem \ref{thm threshold mass}.

\begin{proof}[Proof of Theorem \ref{thm threshold mass}]
For $c>0$ and $\alpha\in\{p,q\}$ let $\kappa_{c,\alpha}:=c^{\frac{1}{d-\frac4\alpha}}$. Define also
$$T_{\ld,\alpha} u(x,y):=\ld^{\frac{2}{\alpha}}u(\ld x,y).$$
Then $u\mapsto T_{\kappa_{c,\alpha},\alpha} u$ defines a bijection between $V(c)$ and $V_{\kappa^2_{c,q}}(1)$ and
$V^{\kappa^2_{c,p}}(1)$ respectively. By using simple scaling arguments one also infers that
$$m_{c}=c^{\frac{d-4/q-2}{d-4/q}}m_{1,\kappa_{c,q}^{2}}=c^{\frac{d-4/p-2}{d-4/p}}m_{1}^{\kappa_{c,p}^{2}}.$$
By same arguments we also deduce that $\wm_{(2\pi)^{-1}c}=c^{\frac{d-4/q-2}{d-4/q}}\wm_{(2\pi)^{-1},\kappa_{c,q}^{2}}
=c^{\frac{d-4/p-2}{d-4/p}}\wm_{(2\pi)^{-1}}^{\kappa_{c,p}^{2}}$ for $c>0$. Notice also that the mapping $c\mapsto \kappa_{c,\alpha}$ is
strictly monotone increasing on $(0,\infty)$. Thus by Lemma \ref{lemma auxiliary} there exists some $c_*,c^*\in(0,\infty)$ such that
\begin{itemize}
\item For all $c\in(0,c_*)$ we have
$$m_{c}=c^{\frac{d-4/p-2}{d-4/p}}m_{1}^{\kappa_{c,p}^2}<c^{\frac{d-4/p-2}{d-4/p}}2\pi \wm_{(2\pi)^{-1}}^{\kappa_{c,p}^2}=2\pi
\wm_{(2\pi)^{-1}c}.$$

\item For all $c\in(c^*,\infty)$ we have
$$m_{c}=c^{\frac{d-4/q-2}{d-4/q}}m_{1,\kappa_{c,q}^2}=c^{\frac{d-4/q-2}{d-4/q}}2\pi \wm_{(2\pi)^{-1},\kappa_{c,q}^2}=2\pi
\wm_{(2\pi)^{-1}c}.$$
\end{itemize}
By the definitions of $c_*$ and $c^*$ it is also clear that $c_*\leq c^*$. This completes the proof of the first part of Theorem \ref{thm
threshold mass}.

It remains to prove the second part of Theorem \ref{thm threshold mass}. In fact, the proof is very similar to the one of the first part, we
hence only give the key steps of the proof without establishing the full details.

Mimicking the proof of Lemma \ref{lemma auxiliary}, we aim to prove the following claim: Let $\gamma_{1,\ld},\gamma_{1}^{\ld},\widehat{\gamma}_{1,\ld},\widehat{\gamma}_1^\ld$ be the quantities defined through \eqref{1.31} and \eqref{1.33}. Then there exist some $0<\ld_*,\ld^*<\infty$ such that
\begin{itemize}
\item[(i)] For all $\ld \in (0,\ld_*)$ we have $\gamma_{1}^\ld<2\pi \widehat{\gamma}_{1}^\ld$ and any minimizer $u^\ld$ of $\gamma_{1}^\ld$ satisfies $\pt_y u^\ld\neq 0$.

\item[(ii)] For all $\ld\in(\ld^*,\infty)$ we have $\gamma_{1,\ld}=2\pi \widehat{\gamma}_{1,\ld}$ and any minimizer $u_\ld$ of $\gamma_{1,\ld}$ satisfies $\pt_y u_\ld= 0$.
\end{itemize}
The proof of (i) follows from a straightforward modification of the proof of Theorem \ref{thm threshold mass}. However, the proof of (ii) is completely different. The main reason here is that the variational problem $\widehat{\gamma}_1^0$ corresponds to a defocusing problem which admits no minimizers, hence we are unable to use a profile $P^0$ in the limiting case as in the proof of Lemma \ref{lemma auxiliary}.

We use a different argument to overcome this issue: Let $\rho$ be the function defined in the proof of Lemma \ref{lemma auxiliary}. For $\ld>0$ also let $P^\ld$ be an optimizer of $\widehat{\gamma}_1^\ld$. In this case, we also define $\psi^\ld:=\rho P^\ld$. Using \eqref{new1} it follows $\mK^\ld(P^\ld)<2\pi\wmK(P^\ld)=0$. Hence arguing as in the proof of Lemma \ref{3.7lem} (by also noticing that $\mu=1$) and using \eqref{new1} we obtain
\begin{align*}
    \gamma_1^\ld &\leq \mS_{1}^{\ld}(\psi^\ld)-\frac{2}{qd}\mK^\ld(\psi^\ld)
    =\frac{\ld}{2} \|\pt_y\rho\|_{L_y^2}^2 \|P^\ld\|_{L_x^2}^2+\|\rho\|_{L_y^2}^2 (\wmS_{1}^{\ld}(P^\ld)-\frac{2}{qd}\wmK^\ld(P^\ld))\\
    &\leq \ld \|\pt_y\rho\|_{L_y^2}^2\widehat{\gamma}_1^\ld+(2\pi-\tilde{\delta})\widehat{\gamma}_1^\ld
    =2\pi\widehat{\gamma}_1^\ld-(\tilde{\delta}-\ld \|\pt_y\rho\|_{L_y^2}^2)\widehat{\gamma}_1^\ld<2\pi\widehat{\gamma}_1^\ld
\end{align*}
for $\ld<\tilde{\delta}\|\pt_y\rho\|_{L_y^2}^{-2}$, where $\tilde{\delta}:=2\pi-\|\pt_y\rho\|_{L_y^2}^2\in(0,2\pi)$. This completes the proof of the claim.

Next, one easily verifies that $u\mapsto T_{\sqrt{\omega}^{-1},\alpha} u$ defines a bijection between the sets $\{u\in H_{x,y}^1:\mK(u)=0\}$ and $\{u\in H_{x,y}^1:\mK_{\omega^{-1}}(u)=0\}$ ($\alpha=q$) respectively $\{u\in H_{x,y}^1:\mK^{\omega^{-1}}(u)=0\}$ ($\alpha=p$). Moreover, we have
\begin{align*}
    \gamma_{\omega}&=\omega^{-\frac{d-4/q-2}{2}}\gamma_{1,\omega^{-1}}=\omega^{-\frac{d-4/p-2}{2}}\gamma_{1}^{\omega^{-1}},\\
    \widehat{\gamma}_{\omega}&=\omega^{-\frac{d-4/q-2}{2}}\widehat{\gamma}_{1,\omega^{-1}}=\omega^{-\frac{d-4/p-2}{2}}\widehat{\gamma}_{1}^{\omega^{-1}}.
\end{align*}
The desired claim then follows by using the same scaling arguments given previously and by noticing that the mapping $\omega\mapsto\omega^{-\frac12}$ is strictly monotone decreasing on $(0,\infty)$. This completes the desired proof.
\end{proof}

\section{Qualitative and quantitative blow-up results: Proof of Theorem \ref{thm-blowup} and \ref{thm-blowup-rate}}\label{sec 4}
This section is devoted to proving the blow-up results Theorem \ref{thm-blowup} and \ref{thm-blowup-rate}. We start with the proof of the qualitative one.

\subsection{Existence of blow-up solutions: Proof of Theorem \ref{thm-blowup}}
We start with some a preliminary lemma.

\begin{lemma}\label{lem:Q-control}
Suppose that the initial datum $u_0$ satisfies
\[
E(u_0)<m_{\|u_0\|_2^2}\quad \text{and}\quad Q(u_0)<0
\]
when $\mu=-1$, or
\[
S_{\omega}(u_0)<\gamma_\omega\quad\text{and}\quad Q(u_0)<0
\]
when $\mu=1$. Then there exists a strictly positive $\delta$ such that $Q(u(t))\leq-\delta$ for any $t\in(-T_{\min}, T_{\max}).$ More precisely,  there exists a constant $\delta'>0$, independent of $t$, such that
\begin{equation}\label{Q-better-control}
Q(u(t))\leq-\delta'\|\nabla_{x,y} u(t)\|_2^2
\end{equation}
for any $t\in(-T_{\min}, T_{\max}).$
\end{lemma}

\begin{proof}
Let us consider the case $\mu=-1$. Firstly, we suppose by the absurd that $Q(u(t))>0$ for some time $t\in(-T_{\min}, T_{\max})$. Then by the
continuity in time of the function $Q(u(t)),$ there exists $\tilde t$ such that $Q(u(\tilde t))=0.$ By definition of the functional $m_c$ and
the conservation of the mass, we have therefore that $m_{\|u_0\|_2^2}\leq E(u(\tilde t))=E(u_0),$ which is a contradiction with respect to
the hypothesis.

For the control away from zero, by means of Lemma \ref{monotoneproperty} we infer the existence of $\tilde\lambda\in(0,1)$ such that
$Q(u_0^{\tilde\lambda})=0$ and
\[
\frac{d}{d\lambda}(E(u^\lambda_0))(\lambda)\geq\frac{d}{d\lambda}(E(u^\lambda_0))(1)=Q(u_0)
\]
for $\lambda\in(\tilde\lambda,1)$. Hence,
\begin{equation}\label{est-Q-control}
\begin{aligned}
E(u_0)&=E(u_0^{\tilde \lambda})+\int_{\tilde\lambda}^{1}\frac{d}{d\lambda}(E(u^\lambda_0))(\lambda)d\lambda\geq
E(u_0^{\tilde\lambda})+(1-\tilde\lambda)\frac{d}{d\lambda}(E(u^\lambda_0))(1)\\
&=E(u_0^{\tilde\lambda})+(1-\tilde\lambda)Q(u_0)>m_{\|u_0\|_2^2}+Q(u_0),
\end{aligned}
\end{equation}
which in turn implies the bound of the Lemma with $\delta=-E(u_0)+m_{\|u_0\|_2^2}$. By the energy conservation, for any other $t$ belonging
to the maximal time of existence, it suffices to repeat the argument above.

Similarly, in the focusing-defocusing case $\mu=1$, if by the absurd we have the existence of $\tilde t$ such that
$Q(u(\tilde t))=0$, then we conclude, by conservation of mass and energy, that $S_{\omega}(u_0)=S_{\omega}(u(\tilde t))\geq \gamma_\omega$,
which contradicts the hypothesis.  Moreover, recalling that the scaling $u^\lambda$ (see \eqref{def of scaling op}) leaves invariant the mass, we have an  estimate analogous
to \eqref{est-Q-control}. Specifically,
\begin{equation}\label{est-Q-control-2}
\begin{aligned}
S_\omega(u_0)&=E(u_0)+\frac{\omega}{2}M(u_0)=E(u_0^{\tilde
\lambda})+\int_{\tilde\lambda}^{1}\frac{d}{d\lambda}(E(u^\lambda_0))(\lambda)d\lambda +\frac{\omega}{2}M(u_0)\\
&\geq
E(u_0^{\tilde\lambda})+(1-\tilde\lambda)\frac{d}{d\lambda}(E(u^\lambda_0))(1)+\frac{\omega}{2}M(u_0)=E(u_0^{\tilde\lambda})+\frac{\omega}{2}M(u_0^{\tilde\lambda})+(1-\tilde\lambda)Q(u_0)\\
&=S_{\omega}(u_0^{\tilde\lambda})+(1-\tilde\lambda)Q(u_0)>\gamma_{\omega}+Q(u_0),
\end{aligned}
\end{equation}
then the claim follows by \eqref{est-Q-control-2} and $\delta=-S_\omega(u_0)+\gamma_{\omega}$.
By the energy conservation, for any other $t$ belonging to the maximal time of existence, it suffices to repeat the argument above.

In order to prove the refined control \eqref{Q-better-control}, recall the following identities for $\mu=1$ and $\mu=-1$, respectively,
\begin{equation}\label{Q-df}
\begin{aligned}
Q(u)&=\|\nabla_x u\|_2^2+\frac{dp}{2(p+2)}\|u\|_{p+2}^{p+2}-\frac{dq}{2(q+2)}\|u\|_{q+2}^{q+2}\\
&=\frac{dq}{2}E(u)+\left(1-\frac{dq}{4}\right)\|\nabla_{x,y} u\|_2^2-\|\partial_yu\|_2^2+\frac{d}{2}\frac{p-q}{p+2}\|u\|_{p+2}^{p+2}
\end{aligned}
\end{equation}
and
\begin{equation}\label{Q-ff}
\begin{aligned}
Q(u)&=\|\nabla_x u\|_2^2-\frac{dp}{2(p+2)}\|u\|_{p+2}^{p+2}-\frac{dq}{2(q+2)}\|u\|_{q+2}^{q+2}\\
&=\frac{dp}{2}E(u)+\left(1-\frac{dp}{4}\right)\|\nabla_{x,y} u\|_2^2-\|\partial_yu\|_2^2+\frac{d}{2}\frac{p-q}{q+2}\|u\|_{q+2}^{q+2}.
\end{aligned}
\end{equation}


Consider first the case $\mu=-1$. By \eqref{Q-ff} and the fact $q>p>\frac4d$, it is
 straightforward to see that
\[
\left(\frac{dp}{4}-1\right)\|\nabla_{x,y} u\|_2^2\leq \frac{dp}{2}E(u)-Q(u),
\]
hence, by means of Lemma \ref{lem:Q-control},   for any $\varepsilon>0$
\[
Q(u)+\varepsilon\left(\frac{dp}{4}-1\right)\|\nabla_{x,y} u\|_2^2\leq \varepsilon\frac{dp}{2}E(u)+(1-\varepsilon)Q(u)\leq
\varepsilon\frac{dp}{2}E(u)-(1-\varepsilon)\delta.
\]
The conclusion follows by taking $\varepsilon$ small enough, recalling the conservation of the energy.\\

As for the defocusing-focusing case, namely $\mu=1$, we repeat the same argument, by using \eqref{Q-df}, Lemma \ref{lem:Q-control},
and again $q>p>\frac 4d$.
\end{proof}


We are now  ready to give the proof of the blow-up results in Theorem \ref{thm-blowup}.
\begin{proof}[Proof of Theorem \ref{thm-blowup}]
Given a smooth function  $\phi: \R^d \to \R$, we
introduce the virial function
\begin{equation}\label{localized-mass}
V_\phi(t):= \int_{\mathbb R^d\times\mathbb T} \phi(x) |u(t,x,y)|^2 dxdy.
\end{equation}

The following identities are nowadays classical (see e.g., \cite{Cazenave2003}), and when no confusion may arise, we omit the domain $\mathbb R^d\times\mathbb T$ along with the
space variables $(x,y)$ to lighten the notation:
\begin{equation}\label{localized-mass-diff}
V'_\phi(t) = 2\ima \int \nabla_x \phi \cdot \nabla_x u(t) \overline{u}(t) dxdy
\end{equation}
and
\begin{equation}\label{virial-identity2}
\begin{aligned}
V''_\phi(t) &= -\int \Delta_x^2 \phi |u(t)|^2 dxdy + 4 \sum_{j,k} \rea \int  \partial^2_{x_j x_k} \phi \partial_{x_j} u(t) \partial_{x_k}
\overline{u}(t) dxdy\\
& +\frac{2\mu p}{p+2} \int \Delta_x \phi |u(t)|^{p+2} dxdy - \frac{2q}{q+2} \int \Delta_x \phi |u(t)|^{q+2} dxdy.
\end{aligned}
\end{equation}

Suppose now that  the initial datum $u_0$ is radial with respect to the  Euclidean variable, i.e., $u_0=u_0(x,y)=u_0(|x|,y)$. Then the  radial symmetry of the solution remains for
any time in the maximal  lifespan. Let $\theta: [0,\infty) \to [0,2]$ be a smooth function satisfying
\begin{align} \label{def-theta}
\theta(r)= \left\{
\begin{array}{ccl}
2 &\text{if} &0\leq r\leq 1, \\
0 &\text{if} & r\geq 2.
\end{array}
\right.
\end{align}
We define the function $\Theta: [0,\infty) \to [0,\infty)$ by
\[
\Theta(r):= \int_0^r \int_0^s \theta(\tau) d\tau ds.
\]
For $\varrho>0$, we define the radial function $\phi_\varrho: \R^d \to \R$ by
\begin{align} \label{phi-R-rad}
\phi_\varrho(x) = \phi_\varrho(r) := \varrho^2 \Theta(r/\varrho), \quad r=|x|.
\end{align}

Straightforward calculations yield 
\begin{equation}\label{vir-fin}
V''_{\phi_\varrho}(t) \leq 8Q(u(t))- \frac{2\mu p}{p+2} \int |k(x)|  |u(t)|^{p+2} dxdy +\frac{2q}{q+2} \int |k(x)| |u(t)|^{q+2}
dxdy+C\varrho^{-2}
\end{equation}
where $k=k(x)=k(|x|)$ is a non-positive radial function supported outside a ball of radius $\varrho$, centered at the origin of $\mathbb R^d$.

At this point we observe that if $\mu=1$ (defocusing-focusing case), the lower order term can be simply estimated by zero, hence we may focus
on the higher order term. Let us denote by \[
\mathfrak{m}(u)(x)=\frac{1}{2\pi}\int_{\mathbb T} u(x,y)dy
\]
the mean of $u$ with respect to the variable on $\mathbb T$, the compact component of the product manifold. By the estimate  $|a+b|^{c}\lesssim
|a|^c + |b|^c$ for $c\geq 1$, and the triangular inequality we estimate the two terms  $\||k|^{1/(q+2)}\mathfrak{m}(u)\|_{q+2}^{q+2}$ and
$\||k|^{1/(q+2)}(u-\mathfrak{m}(u))\|_{q+2}^{q+2}$. As $\mathfrak{m}(u)$ is independent of $y$, by the Strauss embedding (see e.g. \cite{ChoOwaza}), the Minkowski's inequality and the Jensen's inequality, we get
\[
\begin{aligned}
\int |k(x)| |\mathfrak{m}u|^{q+2} dxdy&\lesssim \int |k(x)||\mathfrak{m}(u)(x)|^{q+2}dx\\
&\lesssim \varrho^{-\frac{(d-1)q}{2}}\|\nabla_x \mathfrak{m}(u)\|_{L^2_x}^{q/2}\|\mathfrak{m}(u)\|_{L^2_x}^{q/2}\|\mathfrak{m}(u)\|_{L^2_x}^{2}\\
&\lesssim \varrho^{-\frac{(d-1)q}{2}}\|\nabla_{x,y} u\|_2^{q/2}\|u\|_2^{q/2+2}.
\end{aligned}
\]
The term $\||k|^{1/(q+2)}(u-\mathfrak{m}(u))\|_{q+2}^{q+2}$  is estimated as follows: first recall the  Sobolev embedding
\[
\|u-\mathfrak{m}(u)\|_{L_y^{q+2}}\lesssim \|u\|_{\dot H^{\frac{q}{2(q+2)}}_y},
\]
 see e.g. \cite{sobolev_torus}. Then writing also $u-\mathfrak{m}(u)$ in its Fourier expansion along the $y$-direction we obtain
\[
\begin{aligned}
\||k|^{1/(q+2)}(u-\mathfrak{m}(u))\|_{q+2}^{q+2}&= \int |k(x)||u-\mathfrak{m}(u)|^{q+2}dydx \\
&\lesssim\int\left(   |k(x)|\left( \sum_j |j|^{q/(q+2)}|u_j(x)|^{2}\right)^{\frac{q+2}{2}}\right)dx\\
&=\big\||k|^{1/(q+2)}\||j|^{q/(2q+4)}|u_j|\|_{l^2_j}\big\|_{L^{q+2}_x}^{q+2}.
\end{aligned}
\]
As $2+q\geq2$, we continue by using the Minkowski's inequality,
\[
\begin{aligned}
\||k|^{1/(q+2)}(u-\mathfrak{m}(u))\|_{q+2}^{q+2}&\leq \left(\sum_j |j|^{q/(q+2)}\||k|^{1/(q+2)}|u_j|\|_{L_x^{q+2}}^2\right)^{\frac{q+2}{2}}
\end{aligned}
\]
and by employing again the Strauss embedding theorem as before, we obtain
\[
\begin{aligned}
\||k|^{1/(q+2)}(u-\mathfrak{m}(u))\|_{q+2}^{q+2}&\lesssim \varrho^{-\frac{(d-1)q}{2}}\left(\sum_j |j|^{q/(q+2)} \|\nabla_xu_j\|_{L^2_x}^{q/(q+2)}  \|
u_j\|_{L^2_x}^{(q+4)/(q+2)}\right)^{\frac{q+2}{2}}\\
&=\varrho^{-\frac{(d-1)q}{2}}\left(\sum_j \left(|j|^2\|u_j\|_2^2\right)^{q/(2(q+2))}
\|\nabla_xu_j\|_2^{q/(q+2)}
\|u_j\|_2^{4/(q+2)} \right)^{\frac{q+2}{2}}
\end{aligned}
\]
By means of the H\"older's inequality with exponents $\eta=\frac{2q+4}{q}$, $\gamma=\frac{2q+4}{q}$, and $\mu=\frac{q+2}{2}$, in conjunction with the conservation
of the mass, we end-up with
\[
\begin{aligned}
\||k|^{1/(q+2)}(u-\mathfrak{m}(u))\|_{q+2}^{q+2}
&\lesssim  \varrho^{-\frac{(d-1)q}{2}}\|\nabla_{x,y}
u\|_2^{q}.
\end{aligned}
\]
Then the following virial estimate is  established:
\begin{align} \label{viri-est-rad-1}
V''_{\phi_\varrho}(t) \leq 8 Q(u(t)) +C\varrho^{-2} + C\varrho^{-\frac{(d-1)q}{2}}\left(\|\nabla_{x,y} u\|_2^{q}+\|\nabla_{x,y} u\|_2^{q/2}\right), \quad \forall t\in I_{\max}.
\end{align}
The desired blow-up claim follows then from \eqref{viri-est-rad-1}, \eqref{Q-better-control}, and a convexity argument, provided that $\rho\gg1$.\\

In the focusing-focusing case, i.e. $\mu=-1$, we  cannot  simply estimate the contribution from the lower-order term by using its
non-negativity property. Alternatively, one may verbatim repeat the estimate as for the higher order term.  By doing so, we get  in turn
\begin{equation}\label{viri-est-rad-2}
\begin{aligned}
V''_{\phi_\varrho}(t) 
&\leq 8 Q(u(t)) +o_\rho(1)  +o_\rho(1)\|\nabla_{x,y} u\|_2^{2}, \quad \forall t\in I_{\max},
\end{aligned}
\end{equation}
where we used that $p<q$. 
In both cases we can
conclude with the desired blow up results by taking $\varrho$ sufficiently large.  This completes the proof of Theorem \ref{thm-blowup}.
\end{proof}

\subsection{Proof of Theorem \ref{thm-blowup-rate}}
	In order to prove the blow-up rate results, we are inspired by the scheme introduced by Merle, Rapha\"el, and Szeftel in the context of critical equations, see \cite{MRS}. \medskip

First consider the defocusing-focusing case $\mu=1$. By means of \eqref{Q-df} in conjunction with \eqref{viri-est-rad-1}, we get
	\begin{align*}
	V''_{\phi_\varrho}(t) & \leq 8 Q(u(t)) +C\varrho^{-2} + C \varrho^{-\frac{(d-1)q}{2}}\left(\|\nabla_{x,y} u\|_2^{q}+\|\nabla_{x,y} u\|_2^{q/2}\right)\\
	&\leq 4dq E(u(t)) +(8-2dq)\|\nabla_{x,y} u(t)\|^2_2+ C\varrho^{-2} + C\varrho^{-\frac{(d-1)q}{2}} \left(\|\nabla_{x,y} u\|_2^{q}+\|\nabla_{x,y} u\|_2^{q/2}\right), \quad \forall
t\in I_{\max},
	\end{align*}
	where $\phi_\varrho$ is defined in \eqref{phi-R-rad}. By Young's inequality, we have for any $\varepsilon>0$ and any $t\in I_{\max}$:
	\[
	V''_{\phi_\varrho}(t) \leq 4dq E(u(t)) +(8-2dq+\varepsilon)\|\nabla_{x,y} u(t)\|^2_2 + C\varrho^{-2}  + C_\varepsilon\varrho^{-\frac{(d-1)q}{2-q}}+C_\varepsilon\varrho^{-\frac{2(d-1)q}{4-q}}.
	\]	We select now $\varepsilon$ to be  equal to $dq-4$. Hence the above inequality reduces to
	\begin{align} \label{est-blow-rate-1}
	(dq-4) \|\nabla_{x,y} u(t)\|^2_2 + V''_{\phi_\varrho}(t) \leq  4dq E(u_0)+C\varrho^{-2} + C_\varepsilon\varrho^{-\frac{(d-1)q}{2-q}}+C_\varepsilon\varrho^{-\frac{2(d-1)q}{4-q}}\lesssim \varrho^{-\frac{(d-1)q}{2-q}}, 
	\end{align}
    for $\varrho\ll1$.
	    Now, consider times $0<t_0<t<T_{\max}$. Integrating \eqref{est-blow-rate-1} twice on $(t_0,t)$ gives
	\begin{align*}
	(dq-4)\int_{t_0}^t \int_{t_0}^s \|\nabla_{x,y} u(\tau)\|^2_2 d\tau ds + V_{\phi_\varrho}(t) &\leq C \varrho^{-\frac{(d-1)q}{2-q}} (t-t_0)^2
+ (t-t_0) V'_{\phi_\varrho}(t_0) + V_{\phi_\varrho}(t_0) \\
	&\leq  \varrho^{-\frac{(d-1)q}{2-q}} (t-t_0)^2 + C \varrho(t-t_0) \|\nabla u(t_0)\|_2\\
	&\quad + C\varrho^2,
	\end{align*}
	where we have used the conservation of mass and the estimates below:
	\begin{align*}
	V_{\phi_\varrho}(t_0)&\leq C\varrho^2 \|u(t_0)\|^2_2 \leq C \varrho^2, \\
	 V'_{\phi_\varrho}(t_0)& \leq C \varrho\|\nabla_{x,y} u(t_0)\|_2 \|u(t_0)\|_2 \leq C \varrho \|\nabla_{x,y} u(t_0)\|_2.
	\end{align*}
	  Fubini's Theorem then implies
	  	\[
	\int_{t_0}^t \int_{t_0}^s \|\nabla_{x,y} u(\tau)\|^2_2 d\tau ds = \int_{t_0}^t \left(\int_{\tau}^t ds\right) \|\nabla_{x,y} u(\tau)\|^2_2 d\tau
= \int_{t_0}^t (t-\tau) \|\nabla_{x,y} u(\tau)\|^2_2 d\tau.
	\]
	Recall that $V_{\phi_\varrho}$ is non-negative. Then we get
	\[
	\int_{t_0}^t (t-\tau) \|\nabla_{x,y} u(\tau)\|^2_2 d\tau \leq  \varrho^{-\frac{(d-1)q}{2-q}}(t-t_0)^2 + C \varrho(t-t_0) \|\nabla_{x,y}
u(t_0)\|_2 + C\varrho^2.
	\]
	In the limit $t\to T_{\max}$, by means of the Young's inequality we obtain
	\[
	\begin{aligned}
	\int_{t_0}^{T_{\max}} (T_{\max}-\tau) \|\nabla_{x,y} u(\tau)\|^2_2 d\tau& \leq  \varrho^{-\frac{(d-1)q}{2-q}} (T_{\max}-t_0)^2 + C \varrho(T_{\max}-t_0)
\|\nabla_{x,y} u(t_0)\|_2 + C\varrho^2\\
	&\leq \varrho^{-\frac{(d-1)q}{2-q}} (T_{\max}-t_0)^2 + (T_{\max}-t_0)^2 \|\nabla_{x,y} u(t_0)\|^2_2 + C\varrho^2.
	\end{aligned}
	\]
	Balancing in $\varrho$ by choosing $ \varrho^{-\frac{(d-1)q}{2-q}} (T_{\max}-t_0)^2 = \varrho^2$, or equivalently $\varrho =
(T_{\max}-t_0)^{\frac{2(2-q)}{(d-3)q+4}}$, we  deduce
		\begin{equation}\label{estimate-rho-opt}
		\int_{t_0}^{T_{\max}} (T_{\max}-\tau) \|\nabla_{x,y} u(\tau)\|^2_2 d\tau
		\leq C(T_{\max}-t_0)^{\frac{4(2-q)}{(d-3)q+4}} + (T_{\max}-t_0)^2\|\nabla_{x,y} u(t_0)\|^2_2,
		\end{equation}
for any $0<t_0<T_{\max}$.		By introducing the function
		\begin{equation} \label{5-g}
		g(t):= \int_{t}^{T_{\max}} (T_{\max}-\tau) \|\nabla_{x,y} u(\tau)\|^2_2 d\tau,
		\end{equation}
		from \eqref{estimate-rho-opt} and the Fundamental Theorem of Calculus we  get
		\[
		g(t) \leq C(T_{\max}-t)^{\frac{4(2-q)}{(d-3)q+4}}  - (T_{\max}-t)g'(t), \quad \forall\, 0<t<T_{\max}
		\]
		which  can be straightforwardly rewritten as
		\[
		\frac{d}{dt}\left( \frac{g(t)}{T_{\max}-t} \right) =\frac{1}{(T_{\max}-t)^2} (g(t) + (T_{\max}-t) g'(t)) \leq C(T_{\max}-t)^{\frac{4(2-q)}{(d-3)q+4}-2}.
		\]
		Integrating over the interval $(0,t)$ the above inequality gives
		\[
		\frac{g(t)}{T_{\max}-t} \leq \frac{C}{(T_{\max}-t)^{\frac{(d+1)q-4}{(d-3)q+4}}} \quad \text{ as } \quad t\to T_{\max}^-.
		\]
		Therefore, we have
		\[
	 \int_{t}^{T_{\max}} (T_{\max}-\tau) \|\nabla_{x,y} u(\tau)\|^2_2 d\tau\leq C(T_{\max}-t)^{\frac{4(2-q)}{(d-3)q+4}} \quad \text{ as } \quad t\to
T_{\max}^-.
		\]
The above estimate can be reformulated as 	
		\begin{align} \label{est-g}
		\frac{1}{T_{\max}-t} \int_t^{T_{\max}} (T_{\max}-\tau) \|\nabla_{x,y} u(\tau)\|^2_2 d\tau \leq \frac{C}{(T_{\max}-t)^{\frac{(d+1)q-4}{(d-3)q+4}}}.
		\end{align}
		At this point we consider a sequence $T_n \to T_{\max}^-$, and we note that for any $n$ the function $g$ introduced in \eqref{5-g} is a continuous function on $[T_n,T_{\max}]$ and  differentiable on the interior points
$(T_n,T_{\max})$. Hence,  the mean value theorem gives the  existence of a time $t_n \in (T_n,T_{\max})$ such that the left-hand side of  \eqref{est-g}  satisfies
		\begin{equation*}
		 \frac{1}{T_{\max}-T_n}\int_{T_n}^{T_{\max}} (T_{\max}-\tau)\|\nabla_{x,y} u(\tau)\|^2_2 d\tau=(T_{\max}-t_n)\|\nabla_{x,y} u(t_n)\|^2_2.
		\end{equation*}
		Using \eqref{est-g}, we have
		\[
		  \|\nabla_{x,y} u(t_n)\|_2  \leq \frac{C}{(T_{\max}-t_n)^{\frac{(d-1)q}{(d-3)q+4}}}
		\]
This concludes the proof for the blow-up rate in the case $\mu=1$. \\
	
As for the  focusing-focusing case ($\mu=-1$), using \eqref{Q-ff} instead of \eqref{Q-df}, we see that nothing changes with respect to the
defocusing-focusing case, except for the harmless coefficient of the homogeneous Sobolev norm term. 
We omit the repeating details. This completes the desired proof.


\appendix

\section{Large data scattering: Proof of Theorem \ref{thm-scattering}}\label{Sec: scattering}
In this section, we focus on the large data scattering result for \eqref{nls}, i.e. proving Theorem \ref{thm-scattering}. As already pointed out in the introductory section and also in the recent papers \cite{Luo_inter,Luo_MathAnn_2024} by the second author, the proof of Theorem \ref{thm-scattering} is quite different in the cases $d< 5$ and $d\geq 5$, mainly due to the fact that the nonlinearity becomes less regular in high-dimensional spaces.

The scattering mechanism is adapted from the single-power waveguide arguments in \cite{Luo_inter,Luo_MathAnn_2024}. What has to be verified for \eqref{nls} is that the two nonlinearities can be controlled simultaneously within the same perturbative scheme: the $q$-power determines the scattering norm, while the lower $p$-power is estimated by interpolation using the uniform $H^1$ bound. In the case $d\geq5$ we therefore state the combined-power scattering criterion and local-control estimates needed to enter the interaction Morawetz--Dodson--Murphy (IMDM) argument of \cite{DodsonMurphyNonRadial,Luo_MathAnn_2024}; the steps that are unchanged once these estimates are available are cited from \cite{Luo_MathAnn_2024}. This makes explicit which part of the scattering proof is specific to the simultaneous presence of the $p$- and $q$-terms and which part is inherited from the existing IMDM framework.

Finally, we also note that the proof of Theorem \ref{thm-scattering} in the case $d \leq 4$ can be similarly deduced as in the paper \cite{Luo_inter} via the standard \textit{concentration compactness} method. To keep the paper as concise and short as possible, we omit the latter details.



\subsection{Scattering Criterion}\label{app sub scattering}
First notice that since \eqref{nls} possesses an energy-subcritical nature and we work with a problem in the energy space, a solution of \eqref{nls} can always be extended beyond its lifespan as long as its $H^1$-norm remains bounded. In particular, a solution of \eqref{nls} will be a global solution when certain variational assumption is satisfied (see Section \ref{app sec 2} below). The main issue here is that to guarantee a global solution of \eqref{nls} is also scattering, further control of the solution in infinite time becomes necessary. Such motivation leads to the useful scattering criterion Lemma \ref{scattering criterion}. To formulate Lemma \ref{scattering criterion}, some notion of the exotic Strichartz estimates will also be introduced.

\begin{lemma}[Exotic Strichartz estimates on $\R^d\times\T$, \cite{Luo_MathAnn_2024}]\label{indices}
For any $\alpha\in (\frac{4}{d},\frac{4}{d-1})$ there exist $\ba,\br,\bb,\bs\in(2,\infty)$ such that
\begin{gather*}
(\alpha+1)\bs'=\br,\quad(\alpha+1)\bb'=\ba,\quad\alpha/\br<\min\{1,\frac{2}{d}\},\quad
\frac{2}{\ba}+\frac{d}{\br}=\frac{2}{\alpha}.
\end{gather*}
Moreover, for any $\gamma\in\R$ we have the following exotic Strichartz estimate:
\begin{align}
\|\int_{t_0}^t e^{i(t-s)\dxy}F(s)\,ds\|_{L_t^\ba L_x^\br H^\gamma_y(I)}&\lesssim \|F\|_{L_t^{\bb'} L_x^{\bs'} H^\gamma_y(I)}.
\end{align}
When $d\geq 5$, we can additionally assume that there exists some $0<\beta\ll 1$ such that $\br$ can be chosen as an arbitrary number from $(\frac{\alpha(\alpha+1)d}{\alpha+2},\frac{\alpha(\alpha+1)d}{\alpha+2}+\beta)$.
\end{lemma}

\begin{lemma}[Scattering criterion, \cite{Luo_MathAnn_2024}]\label{scattering criterion}
Let $u$ be a global solution of \eqref{nls} and assume that
$$\|u\|_{L_t^\infty H_{x,y}^1(\R)}\leq A.$$
Then for any $\sigma>0$ there exist
$\vare=\vare(\sigma,A)$ sufficiently small and $T_0=T_0(\sigma,\vare,A)$ sufficiently large such that if for all $a\in\R$ there exists
$T\in(a,a+T_0)$ such that $[T-\vare^{-\sigma},T]\subset(a,a+T_0)$ and
\begin{align}
\|u\|_{L_t^{\ba_q} L_x^{\br_q} H_y^s(T-\vare^{-\sigma},T)} \lesssim
\vare^\mu\label{4.1}
\end{align}
for some $\mu>0$, where $(\ba_q,\br_q)$ is the exotic-admissible pair in Lemma \ref{indices} corresponding to the exponent $q$, then $u$ scatters forward in time.
\end{lemma}
\begin{proof}
This follows from the proof of \cite[Lem. 6]{Luo_MathAnn_2024} after applying the Duhamel estimates separately to the $p$- and $q$-terms. The $q$-term is controlled in the exotic Strichartz norm appearing in \eqref{4.1}; since $p<q$ and both exponents are intercritical, the corresponding estimate for the $p$-term is obtained by interpolation between this control and the uniform $L_t^\infty H_{x,y}^1$ bound. Thus both nonlinearities satisfy the same perturbative smallness requirement used in the cited scattering argument.
\end{proof}

The scattering criterion will be applied in conjunction with the following useful local control result. The latter is obtained from the proof of \cite[Lem. 5]{Luo_MathAnn_2024} by treating the two nonlinearities separately and using the same interpolation principle: the higher $q$-power fixes the admissible control norm, while the lower $p$-power is interpolated with the uniform energy bound.

\begin{lemma}[Local control of a solution, \cite{Luo_MathAnn_2024}]\label{local control}
Let $u$ be a global solution of \eqref{nls} with $\|u\|_{L_t^\infty H_{x,y}^1(\R)}<\infty$ and let $s\in(\frac{1}{2},1-s_q)$, where
$s_q=\frac{d}{2}-\frac{2}{q}$. Then for any $L_x^2$-admissible pair $(\ell_1,\ell_2)$ (namely $(\ell_1,\ell_2)$ satisfies $\frac{2}{\ell_1}+\frac{d}{\ell_2}=\frac{d}{2}$) with $(\ell_1,\ell_2,d)\neq (2,\infty,2)$ we have
\begin{align}
\|u\|_{L_t^{\ell_1} W_x^{1-s,\ell_2}H_y^s(I)}\lesssim \la I\ra^{\frac{1}{\ell_1}}.
\end{align}
\end{lemma}

\subsection{Variational analysis and the IMDM-estimates}\label{app sec 2}
As mentioned in Subsection \ref{app sub scattering}, we will need to establish the uniform $H^1$-boundedness for the solution of \eqref{nls} by appealing to suitable variational arguments that can be deduced nowadays in a quite standard way by using the functional inequalities such as the Gagliardo-Nirenberg or Sobolev inequalities. In the context of the waveguide setting, suitable scale-invariant (w.r.t. the $x$-direction) replacement of such functional inequalities becomes necessary, see e.g. Lemma \ref{lemma gn additive}. Moreover, to fit the non-local nature of the Morawetz inequalities of interaction type, additional spatial translation shall also be taken into account in the functional inequalities, see \cite[Lem. 2.1]{DodsonMurphyNonRadial}. All the consideration leads to the following coercivity result. For a proof, see e.g. \cite{Luo_MathAnn_2024,InteractionCombined}.

\begin{lemma}[Coercivity property, \cite{Luo_MathAnn_2024,InteractionCombined}]\label{lem coercive}
Let $u$ be a solution of \eqref{nls} with $u(0)\in\mathcal{A}$, where
\begin{align*}
\left\{
\begin{array}{ll}
\mathcal{A}:=\{u\in S(c):\mH(u)<m_c,\,\mK(u)>0\},&\text{when $\mu=-1$},\\
\mathcal{A}:=\{u\in H_{x,y}^1\setminus \{0\}:\mS_\omega(u)<\gamma_\omega,\,\mK(u)>0\},&\text{when $\mu=1$}.
\end{array}
\right.
\end{align*}
Then $u$ is global and $u(t)\in\mathcal{A}$ for all $t\in\R$. Moreover, there exist $0<\delta\ll 1$ and $R_0\gg 1$ such that for all $R\geq R_0$, $z\in \R^d$, $t\in\R$ we have
\begin{align}
\mK(\chi_R(\cdot-z)u^\xi(t))\geq \delta\|\nabla_x(\chi_R(\cdot-z))u^\xi(t)\|^2_{2},\label{5.2}
\end{align}
where $u^\xi(t,x,y):=e^{ix\cdot\xi}u(t,x,y)$ and
\begin{equation*}
\xi=\xi(t,z,R)=\left\{
\begin{array}{rr}
             -\frac{\int \mathrm{Im}(\chi_R^2(x-z)\bar{u}(t,x,y)\nabla_x u(t,x,y))\,dxdy}{\int \chi_R^2(x-z)|u(t,x,y)|^2\,dxdy}, &
             \text{if }\int \chi_R^2(x-z)|u(t,x,y)|^2\,dxdy\neq 0,\\
             0, &\text{if }\int \chi_R^2(x-z)|u(t,x,y)|^2\,dxdy=0.
             \end{array}
\right.
\end{equation*}
\end{lemma}

Having established the previous coercivity property of the solution $u$ of \eqref{nls}, we may immediately deduce the following IMDM-inequality by using the interaction Morawetz potential initiated by Dodson and Murphy \cite{Dodson4dmassfocusing,DodsonMurphyNonRadial} which plays a crucial role in the rest of the paper. In the waveguide setting, we shall simply apply the interaction Morawetz inequality along the $x$-direction and integrate the quantities over the $y$-direction without taking other operations. This can also be seen from the observation that the NLS only possesses dispersive effects in the infinite Euclidean space.

\begin{lemma}[IMDM-inequality, \cite{DodsonMurphyNonRadial,Luo_MathAnn_2024}]\label{lem inter morawetz}
Let $u$ be a global solution of \eqref{nls} satisfying the assumptions in Lemma \ref{lem coercive}. Then for any $\vare>0$ there exist $T_0=T_0(\vare)\gg 1$, $J=J(\vare)\gg 1$, $R_0=R_0(\vare,u_0)\gg 1$ and $\eta=\eta(\vare)\ll 1$ such that for any
$a\in\R$ we have
\begin{equation}\label{5.12}
\begin{aligned}
&\,\frac{1}{JT_0}\int_a^{a+T_0}\int_{R_0}^{R_0e^J}\frac{1}{R^d}\int_{(\R_{x_a}^d\times\T_{y_a})\times (\R_{x_b}^d\times\T_{y_b})\times\R^d_z}
\\
&\qquad |\chi_R(x_b-z)u(t,x_b,y_b)|^2|\nabla_x(\chi_R(x_a-z)u^\xi(t,x_a,y_a))|^2d(x_a,y_a)d(x_b,y_b)dz\frac{dR}{R}dt\lesssim \vare.
\end{aligned}
\end{equation}
\end{lemma}
\subsection{Conclusion}
In this final subsection, we explain briefly how we are able to prove Theorem \ref{thm-scattering} in the case $d\geq 5$ by using the lemmas stated in previous subsections.

Notice first that by the scattering criterion (Lemma \ref{scattering criterion}) we will need to prove
$$\|u\|_{L_t^{\ba} L_x^{\br} H_y^s(T-\vare^{-\sigma},T)} \lesssim
\vare^\mu$$
for some $\mu>0$, where $(\ba,\br)=(\ba_q,\br_q)$ is the exotic-admissible pair in Lemma \ref{indices} corresponding to the exponent $q$. To deduce this, we firstly apply the pigeonhole principle to reduce the (both spatially and temporally) averaged inequality \eqref{5.12} to the localized form
\begin{align}
&\,\int_{t_0-\vare^{-\sigma}}^{t_0}\sum_{w\in\Z^d}
\|\chi_{R_1}(\cdot-\frac{R_1}{4}(w+\theta_0))u(t)\|_{L_{x,y}^2}^2\|\nabla_x(\chi_{R_1}(\cdot-\frac{R_1}{4}(w+\theta_0))u^\xi(t))\|_{L_{x,y}^2}^2 dt\lesssim \vare^{1-\sigma},
\end{align}
where $\theta_0:\Z^d\to [0,1]^d$ is a suitable function whose existence is guaranteed by the mean value theorem. Using the modified Gagliardo-Nirenberg inequality on $\R^d$ (see \cite[Lem. 2.1]{DodsonMurphyNonRadial})
\begin{align}
\|u\|^2_{L_x^{\frac{2d}{d-1}}}\lesssim\|u\|_{L_x^{2}}\|\nabla_x u^\xi\|_{L_x^{2}}\label{6.5}
\end{align}
and H\"older and Minkowski inequalities we obtain
\begin{align}
\int_{t_0-\vare^{-\sigma}}^{t_0}\sum_{w\in\Z^d}\|\chi_{R_1}(\cdot-\frac{R_1}{4}(w+\theta_0))u(t)\|^4_{L_x^{\frac{2d}{d-1}}L_y^2}\,dt
\lesssim \vare^{1-\sigma}.\label{6.7}
\end{align}
Another application of the H\"older, Cauchy-Schwarz inequalities and the Sobolev embedding $H_x^1\hookrightarrow L_x^{\frac{2d}{d-2}}$ yields
\begin{align*}
\sum_{w\in\Z^d}\|\chi_{R_1}(\cdot-\frac{R_1}{4}(w+\theta_0))u(t)\|^2_{L_x^{\frac{2d}{d-1}}L_y^2}
\lesssim\|u(t)\|_{H_{x,y}^1}^2+O(R_1^{-2}\eta^{-2})\|u(t)\|_{L_x^2}^2\lesssim 1
\end{align*}
by choosing $R_1\gg 1$, hence
\begin{align}
\int_{t_0-\vare^{-\sigma}}^{t_0}\sum_{w\in\Z^d}\|\chi_{R_1}(\cdot-\frac{R_1}{4}(w+\theta_0))u(t)\|^2_{L_x^{\frac{2d}{d-1}}L_y^2}\,dt
\lesssim \vare^{-\sigma}.\label{6.8}
\end{align}
Interpolating \eqref{6.7} and \eqref{6.8} we obtain
\begin{equation}\label{a10}
\begin{aligned}
\|u\|^{\frac{2d}{d-1}}_{L_{t,x}^{\frac{2d}{d-1}}L_y^2(t_0-\vare^{-\sigma},t_0)}&
\lesssim
\bg(\int_{t_0-\vare^{-\sigma}}^{t_0}\sum_{w\in\Z^d}
\|\chi_{R_1}(\cdot-\frac{R_1}{4}(w+\theta_0))u(t)\|^{4}_{L_x^{\frac{2d}{d-1}}L_y^2}\,dt\bg)^{\frac{1}{d-1}}
\\
&\quad \times\bg(\int_{t_0-\vare^{-\sigma}}^{t_0}\sum_{w\in\Z^d}
\|\chi_{R_1}(\cdot-\frac{R_1}{4}(w+\theta_0))u(t)\|^{2}_{L_x^{\frac{2d}{d-1}}L_y^2}\,dt\bg)^{\frac{d-2}{d-1}}
\lesssim \vare^{\frac{1}{d-1}-\sigma}.
\end{aligned}
\end{equation}
Now the desired claim follows from \eqref{a10}, Lemma \ref{local control} (setting $|I|=t_0-(t_0-\vare^{-\sigma})=\vare^{-\sigma}$ therein) and suitable interpolation (noticing that $(\frac{2d}{d-1},\frac{2d}{d-1})$ is not yet an $L_x^2$-admissible pair, hence Lemma \ref{local control} is not directly applicable).

This essentially explains the idea for proving Theorem \ref{thm-scattering} in the case $d\geq 5$ by making use of the IMDM-estimates. For full details, we refer to \cite{Luo_MathAnn_2024}.

\section{The grow-up result: Proof of Theorem \ref{thm-growup}}\label{App-grow}

In this section we give the proof of the grow-up result Theorem \ref{thm-growup}. Consider \eqref{localized-mass} with a cut-off function $\vartheta:\mathbb R^+\mapsto[0,1]$ with $0\leq \vartheta'\leq 4$ and
\begin{align} \label{def-vartheta2}
\vartheta(|x|)= \left\{
\begin{array}{ccl}
0 &\text{if} &  |x|\leq 1, \\
1 &\text{if} & |x|\geq 2.
\end{array}
\right.
\end{align}
Consider $\vartheta_\varrho=\vartheta(|x|/\varrho)$. Recall the definition of $V_{\vartheta_\varrho}(t)$ as in \eqref{localized-mass}.  If we suppose that the solution is global and $\displaystyle \sup_{\mathbb R^+}\|\nabla_{x,y}u(t)\|_2$ is finite, then by the conservation of mass, there exists a constant $C>0$ such that
\[
V_{\vartheta_\rho}(t)=V_{\vartheta_\varrho}(0)+\int_0^tV_{\vartheta_\varrho}'(s)\,ds\leq o_\rho(1)+Ct\varrho^{-1},
\]
where  $V_{\vartheta_\varrho}(0)=o_\varrho(1)$ as $\varrho\to+\infty$ by means of the the dominated convergence theorem. Clearly
\[
\int_{\{|x|\geq\varrho\}\times\mathbb T}|u(x,y,t)|^2\,dxdy\leq V_{\vartheta_\varrho}(t),
\]
so we have that for any $\tilde \delta>0$
\begin{equation}\label{prop-speed}
\int_{\{|x|\geq\varrho\}\times\mathbb T}|u(x,y,t)|^2dxdy\leq o_\varrho(1)+\tilde\delta, \quad \text{ for } \quad  t\leq \tilde T:=C^{-1}\varrho\tilde\delta.
\end{equation}
Going back to  \eqref{localized-mass} and \eqref{virial-identity2} with the localization function $\phi_\rho$ as defined in \eqref{phi-R-rad}, and by recalling that  for a radial function
\[
\partial_{x_j} = \frac{x_j}{r} \partial_r, \quad \partial^2_{x_jx_k} = \left( \frac{\delta_{x_jx_k}}{r} - \frac{x_jx_k}{r^3} \right) \partial_r + \frac{x_j x_k}{r^2} \partial^2_r,
\]
we have
\begin{equation}\label{virial-identity-growup}
\begin{aligned}
V''_{\phi_\varrho}(t) &= 8Q(u(t))-\int \Delta_x^2 \phi_\varrho |u(t)|^2 dxdy\\
&+4\int\left(\frac{\phi_\varrho'}{r}-2\right)|\nabla_xu|^2 dxdy +4\int\left(\frac{\phi_\varrho''}{r^2}-\frac{\phi_\varrho'}{r^3}\right)|x\cdot\nabla_xu|^2 dxdy\\
&+\frac{2\mu p}{p+2} \int \left(\phi_\varrho''-(d-1)\frac{\phi_\varrho}{r}-2d\right) |u(t)|^{p+2} dxdy \\
&- \frac{2q}{q+2} \int \left(\phi_\varrho''-(d-1)\frac{\phi_\varrho}{r}-2d\right)|u(t)|^{q+2} dxdy.
\end{aligned}
\end{equation}
By the support properties of the localization function and  by interpolation it follows that
\begin{equation}\label{viri-est-rad-bis}
V''_{\phi_\varrho}(t)\leq 8Q(u(t))+C_p\|u\|^{\eta_p(p+2)}_{L^2(\{|x|\geq \varrho\}\times\mathbb T)}+C_q\|u\|^{\eta_q(q+2)}_{L^2(\{|x|\geq \varrho\}\times\mathbb T)}
\end{equation}
where $\eta_p,\eta_q\in(0,1)$. Thus, by combining  Lemma \ref{lem:Q-control}, \eqref{prop-speed}, \eqref{viri-est-rad-bis} we obtain
\[
V''_{\phi_\varrho}(t)\leq -8\delta+2\max\{C_p,C_q\}\left(  o_\varrho(1)+\tilde\delta^{\min\{\eta_p(p+2),\eta_q(q+2)\}}  \right) \quad \text{ for } \quad t\leq \tilde T.
\]
A choice $\tilde\delta\ll1$ implies that for
$\varrho$ large enough
\begin{equation}\label{v-primeprime-small-o}
V''_{\phi_\varrho}(t)\leq -4\delta<0.
\end{equation}
As $\phi\leq|x|^2$, one notes that
\begin{equation}\label{v-small-o}
\begin{aligned}
V_{\phi_\varrho}(0)&\leq\int_{\{|x|\leq\sqrt\varrho\}\times\mathbb T}|x|^2|u(x,y,0)|^2dxdy+\int_{\{\sqrt \varrho\leq |x|\leq2\varrho\}\times\mathbb T}|x|^2|u(x,y,0)|^2dxdy\\
&\leq\varrho M(u_0)+4\varrho^2 o_\varrho(1)=Co_\varrho(1)\varrho^2.
\end{aligned}
\end{equation}
Similarly,
\begin{equation}\label{v-prime-small-o}
V'_{\phi_\varrho}(0)\leq Co_\varrho(1)\varrho.
\end{equation}
It easily follows, by integrating twice in time over the interval $[0,\tilde T]$ the inequality \eqref{v-primeprime-small-o} and by using \eqref{v-small-o}, \eqref{v-prime-small-o}, along with the definition of $\tilde T\sim \varrho$,  we conclude with
\[
V_{\phi_{\varrho}}(\tilde T)=V_{\phi_{\varrho}}(0)+\tilde T V'_{\phi_{\varrho}}(0)+\int_0^{\tilde T}\int_0^sV''_{\phi_{\varrho}}(s)dsdt\leq C( o_\varrho(1)-\delta)\varrho^2\leq -\frac{C\delta}{2}\varrho^2,
\]
which is a contradiction with respect to the non-negativity of the function $V_{\phi_\varrho}(t)$.


\begin{thebibliography}{10}

\bibitem{InteractionCombined}
{\sc Bellazzini, J., Dinh, V.~D., and Forcella, L.}
\newblock Scattering for {N}onradial 3{D} {NLS} with {C}ombined
  {N}onlinearities: {T}he {I}nteraction {M}orawetz {A}pproach.
\newblock {\em SIAM J. Math. Anal. 56}, 3 (2024), 3110--3143.

\bibitem{BellazziniJeanjean2016}
{\sc Bellazzini, J., and Jeanjean, L.}
\newblock On dipolar quantum gases in the unstable regime.
\newblock {\em SIAM J. Math. Anal. 48}, 3 (2016), 2028--2058.

\bibitem{Bellazzini2013}
{\sc Bellazzini, J., Jeanjean, L., and Luo, T.}
\newblock Existence and instability of standing waves with prescribed norm for
  a class of {S}chr\"odinger-{P}oisson equations.
\newblock {\em Proc. Lond. Math. Soc. (3) 107}, 2 (2013), 303--339.

\bibitem{sobolev_torus}
{\sc B\'{e}nyi, A., and Oh, T.}
\newblock The {S}obolev inequality on the torus revisited.
\newblock {\em Publ. Math. Debrecen 83}, 3 (2013), 359--374.

\bibitem{Berestycki1983}
{\sc Berestycki, H., and Lions, P.-L.}
\newblock Nonlinear scalar field equations. {II}. {E}xistence of infinitely
  many solutions.
\newblock {\em Arch. Rational Mech. Anal. 82}, 4 (1983), 347--375.

\bibitem{BrezisKato}
{\sc Br\'{e}zis, H., and Kato, T.}
\newblock Remarks on the {S}chr\"{o}dinger operator with singular complex
  potentials.
\newblock {\em J. Math. Pures Appl. (9) 58}, 2 (1979), 137--151.

\bibitem{carles2020soliton}
{\sc Carles, R., Klein, C., and Sparber, C.}
\newblock On ground state (in-)stability in multi-dimensional cubic-quintic
  {Schr{\"o}dinger} equations.
\newblock {\em ESAIM, Math. Model. Numer. Anal. 57}, 2 (2023), 423--443.

\bibitem{Carles_Sparber_2021}
{\sc Carles, R., and Sparber, C.}
\newblock Orbital stability vs. scattering in the cubic-quintic
  {S}chr\"{o}dinger equation.
\newblock {\em Rev. Math. Phys. 33}, 3 (2021), 2150004, 27.

\bibitem{Cazenave2003}
{\sc Cazenave, T.}
\newblock {\em Semilinear {S}chr\"odinger equations}, vol.~10 of {\em Courant
  Lecture Notes in Mathematics}.
\newblock New York University, Courant Institute of Mathematical Sciences, New
  York; American Mathematical Society, Providence, RI, 2003.

\bibitem{Cheng2020}
{\sc Cheng, X.}
\newblock Scattering for the mass super-critical perturbations of the mass
  critical nonlinear {S}chr\"{o}dinger equations.
\newblock {\em Illinois J. Math. 64}, 1 (2020), 21--48.

\bibitem{CubicR2T1Scattering}
{\sc Cheng, X., Guo, Z., Yang, K., and Zhao, L.}
\newblock On scattering for the cubic defocusing nonlinear {S}chr\"{o}dinger
  equation on the waveguide {$\mathbb R^2 \times \mathbb T$}.
\newblock {\em Rev. Mat. Iberoam. 36}, 4 (2020), 985--1011.

\bibitem{R1T1Scattering}
{\sc Cheng, X., Guo, Z., and Zhao, Z.}
\newblock On scattering for the defocusing quintic nonlinear {S}chr\"{o}dinger
  equation on the two-dimensional cylinder.
\newblock {\em SIAM J. Math. Anal. 52}, 5 (2020), 4185--4237.

\bibitem{MiaoDoubleCrit}
{\sc Cheng, X., Miao, C., and Zhao, L.}
\newblock Global well-posedness and scattering for nonlinear {S}chr\"{o}dinger
  equations with combined nonlinearities in the radial case.
\newblock {\em J. Differential Equations 261}, 6 (2016), 2881--2934.

\bibitem{Cheng_JMAA}
{\sc Cheng, X., Zhao, Z., and Zheng, J.}
\newblock Well-posedness for energy-critical nonlinear {S}chr\"{o}dinger
  equation on waveguide manifold.
\newblock {\em J. Math. Anal. Appl. 494}, 2 (2021), Paper No. 124654, 14.

\bibitem{ChoOwaza}
{\sc Cho, Y., and Ozawa, T.}
\newblock Sobolev inequalities with symmetry.
\newblock {\em Commun. Contemp. Math. 11}, 3 (2009), 355--365.

\bibitem{GrossPitaevskiR1T1}
{\sc de~Laire, A., Gravejat, P., and Smets, D.}
\newblock Minimizing travelling waves for the gross-pitaevskii equation on
  $\mathbb{R} \times \mathbb{T}$, 2022.

\bibitem{Dodson4dmassfocusing}
{\sc Dodson, B.}
\newblock Global well-posedness and scattering for the mass critical nonlinear
  {S}chr\"{o}dinger equation with mass below the mass of the ground state.
\newblock {\em Adv. Math. 285\/} (2015), 1589--1618.

\bibitem{DodsonMurphyNonRadial}
{\sc Dodson, B., and Murphy, J.}
\newblock A new proof of scattering below the ground state for the non-radial
  focusing {NLS}.
\newblock {\em Math. Res. Lett. 25}, 6 (2018), 1805--1825.

\bibitem{DWZ}
{\sc Du, D., Wu, Y., and Zhang, K.}
\newblock On blow-up criterion for the nonlinear {S}chr\"{o}dinger equation.
\newblock {\em Discrete Contin. Dyn. Syst. 36}, 7 (2016), 3639--3650.

\bibitem{Ghoussoub1993}
{\sc Ghoussoub, N.}
\newblock {\em Duality and perturbation methods in critical point theory},
  vol.~107 of {\em Cambridge Tracts in Mathematics}.
\newblock Cambridge University Press, Cambridge, 1993.
\newblock With appendices by David Robinson.

\bibitem{Luo_SIMA_2024}
{\sc Hajaiej, H., Luo, Y., and Song, L.}
\newblock On existence and stability results for normalized ground states of
  mass-subcritical biharmonic nonlinear {S}chr\"odinger equation on
  {$\mathbb{R}^d\times \mathbb{T}^n$}.
\newblock {\em SIAM J. Math. Anal. 56}, 4 (2024), 4415--4439.

\bibitem{HaniPausader}
{\sc Hani, Z., and Pausader, B.}
\newblock On scattering for the quintic defocusing nonlinear {S}chr\"{o}dinger
  equation on {$\mathbb R\times\mathbb T^2$}.
\newblock {\em Comm. Pure Appl. Math. 67}, 9 (2014), 1466--1542.

\bibitem{Ionescu2}
{\sc Ionescu, A.~D., and Pausader, B.}
\newblock Global well-posedness of the energy-critical defocusing {NLS} on
  {$\mathbb R\times \mathbb T^3$}.
\newblock {\em Comm. Math. Phys. 312}, 3 (2012), 781--831.

\bibitem{Jeanjean1997}
{\sc Jeanjean, L.}
\newblock Existence of solutions with prescribed norm for semilinear elliptic
  equations.
\newblock {\em Nonlinear Anal. 28}, 10 (1997), 1633--1659.

\bibitem{waveguide_ref_3}
{\sc Kengne, E., Vaillancourt, R., and Malomed, B.~A.}
\newblock Bose{\textendash}einstein condensates in optical lattices: the
  cubic{\textendash}quintic nonlinear {S}chrödinger equation with a periodic
  potential.
\newblock {\em Journal of Physics B: Atomic, Molecular and Optical Physics 41},
  20 (2008), 205202.

\bibitem{killip2020cubicquintic}
{\sc Killip, R., Murphy, J., and Visan, M.}
\newblock Scattering for the cubic-quintic {NLS}: crossing the virial
  threshold.
\newblock {\em SIAM J. Math. Anal. 53}, 5 (2021), 5803--5812.

\bibitem{killip_visan_soliton}
{\sc Killip, R., Oh, T., Pocovnicu, O., and Vi\c{s}an, M.}
\newblock Solitons and scattering for the cubic-quintic nonlinear
  {S}chr\"{o}dinger equation on {$\mathbb{R}^3$}.
\newblock {\em Arch. Ration. Mech. Anal. 225}, 1 (2017), 469--548.

\bibitem{Liouville2}
{\sc Li, Y., and Zhang, L.}
\newblock Liouville-type theorems and {H}arnack-type inequalities for
  semilinear elliptic equations.
\newblock {\em J. Anal. Math. 90\/} (2003), 27--87.

\bibitem{Luo_inter}
{\sc Luo, Y.}
\newblock Normalized ground states and threshold scattering for focusing {NLS}
  on $\mathbb{R}^d\times\mathbb{T}$ via semivirial-free geometry, 2022.

\bibitem{Luo_JFA_2022}
{\sc Luo, Y.}
\newblock Sharp scattering for the cubic-quintic nonlinear {S}chr\"{o}dinger
  equation in the focusing-focusing regime.
\newblock {\em J. Funct. Anal. 283}, 1 (2022), Paper No. 109489, 34.

\bibitem{Luo_energy_crit}
{\sc Luo, Y.}
\newblock On long time behavior of the focusing energy-critical {NLS} on
  {$\mathbb{R}^d\times \mathbb{T}$} via semivirial-vanishing geometry.
\newblock {\em J. Math. Pures Appl. (9) 177\/} (2023), 415--454.

\bibitem{Luo_JFA_2024}
{\sc Luo, Y.}
\newblock Almost sure scattering for the defocusing cubic nonlinear
  {S}chr\"odinger equation on {$\mathbb{R}^3\times \mathbb{T}$}.
\newblock {\em J. Funct. Anal. 287}, 4 (2024), Paper No. 110492, 33.

\bibitem{Luo_JGA_2024}
{\sc Luo, Y.}
\newblock A {L}egendre--{F}enchel {I}dentity for the {N}onlinear
  {S}chr\"odinger {E}quations on {$\mathbb R^d\times\mathbb T^m$}: {T}heory and
  {A}pplications.
\newblock {\em J. Geom. Anal. 34}, 10 (2024), Paper No. 313.

\bibitem{Luo_MathAnn_2024}
{\sc Luo, Y.}
\newblock Sharp scattering for focusing intercritical {NLS} on high-dimensional
  waveguide manifolds.
\newblock {\em Math. Ann. 389}, 1 (2024), 63--83.

\bibitem{Luo_DoubleCritical}
{\sc Luo, Y.}
\newblock On sharp scattering threshold for the mass-energy double critical
  {NLS} via double track profile decomposition.
\newblock {\em Ann. Inst. H. Poincar\'{e} C Anal. Non Lin\'{e}aire\/} (to
  appear).

\bibitem{MRS}
{\sc Merle, F., Rapha\"{e}l, P., and Szeftel, J.}
\newblock On collapsing ring blow-up solutions to the mass supercritical
  nonlinear {S}chr\"{o}dinger equation.
\newblock {\em Duke Math. J. 163}, 2 (2014), 369--431.

\bibitem{Murphy2021CPDE}
{\sc Murphy, J.}
\newblock Threshold scattering for the 2d radial cubic-quintic {NLS}.
\newblock {\em Comm. Partial Differential Equations\/} (2021), 1--22.

\bibitem{waveguide_ref_1}
{\sc Schneider, T.}
\newblock {\em Nonlinear Optics in Telecommunications}.
\newblock Springer Science \& Business Media, Berlin Heidelberg, 2013.

\bibitem{waveguide_ref_2}
{\sc Snyder, A., and Love, J.}
\newblock {\em Optical Waveguide Theory}.
\newblock Springer Science \& Business Media, Berlin Heidelberg, 2012.

\bibitem{Struwe1996}
{\sc Struwe, M.}
\newblock {\em Variational methods}, second~ed., vol.~34 of {\em Ergebnisse der
  Mathematik und ihrer Grenzgebiete (3) [Results in Mathematics and Related
  Areas (3)]}.
\newblock Springer-Verlag, Berlin, 1996.
\newblock Applications to nonlinear partial differential equations and
  Hamiltonian systems.

\bibitem{TaoVisanZhang}
{\sc Tao, T., Visan, M., and Zhang, X.}
\newblock The nonlinear {S}chr\"{o}dinger equation with combined power-type
  nonlinearities.
\newblock {\em Comm. Partial Differential Equations 32}, 7-9 (2007),
  1281--1343.

\bibitem{TTVproduct2014}
{\sc Terracini, S., Tzvetkov, N., and Visciglia, N.}
\newblock The nonlinear {S}chr\"{o}dinger equation ground states on product
  spaces.
\newblock {\em Anal. PDE 7}, 1 (2014), 73--96.

\bibitem{TNCommPDE}
{\sc Tzvetkov, N., and Visciglia, N.}
\newblock Small data scattering for the nonlinear {S}chr\"{o}dinger equation on
  product spaces.
\newblock {\em Comm. Partial Differential Equations 37}, 1 (2012), 125--135.

\bibitem{TzvetkovVisciglia2016}
{\sc Tzvetkov, N., and Visciglia, N.}
\newblock Well-posedness and scattering for nonlinear {S}chr\"{o}dinger
  equations on {$\mathbb{R}^d\times\mathbb{T}$} in the energy space.
\newblock {\em Rev. Mat. Iberoam. 32}, 4 (2016), 1163--1188.

\bibitem{defocusing5dandhigher}
{\sc Visan, M.}
\newblock The defocusing energy-critical nonlinear {S}chr\"{o}dinger equation
  in higher dimensions.
\newblock {\em Duke Math. J. 138}, 2 (2007), 281--374.

\bibitem{YuYueZhao2021}
{\sc Yu, X., Yue, H., and Zhao, Z.}
\newblock Global {W}ell-posedness for the focusing cubic {NLS} on the product
  space {$\mathbb{R} \times \mathbb{T}^3$}.
\newblock {\em SIAM J. Math. Anal. 53}, 2 (2021), 2243--2274.

\bibitem{RmT1}
{\sc Zhao, Z.}
\newblock On scattering for the defocusing nonlinear {S}chr\"{o}dinger equation
  on waveguide {$\mathbb R^m\times \mathbb T$} (when {$m = 2,3$}).
\newblock {\em J. Differential Equations 275\/} (2021), 598--637.

\bibitem{ZhaoZheng2021}
{\sc Zhao, Z., and Zheng, J.}
\newblock Long time dynamics for defocusing cubic nonlinear {S}chr\"{o}dinger
  equations on three dimensional product space.
\newblock {\em SIAM J. Math. Anal. 53}, 3 (2021), 3644--3660.

\end{thebibliography}

\end{document}